\documentclass[journal]{IEEEtran}

\usepackage{amsmath,amsthm,amssymb,amsfonts,enumerate}
\usepackage{cite}
\usepackage{xspace}
\usepackage{psfrag}
\usepackage{tikz}
\usepackage{color,xcolor}
\usepackage{verbatim}
\usepackage{tikz}
\usepackage{epstopdf}
\usepackage{soul}
\usepackage[normalem]{ulem}

\newtheorem{theorem}{Theorem}[section]
\newtheorem{corollary}[theorem]{Corollary}
\newtheorem{lemma}[theorem]{Lemma}

\newtheorem{proposition}[theorem]{Proposition}
\newtheorem{definition}[theorem]{Definition}

\newtheorem{exx}[theorem]{Example}
\newtheorem{remm}[theorem]{Remark}

\def\BibTeX{{\rm B\kern-.05em{\sc i\kern-.025em b}\kern-.08em
    T\kern-.1667em\lower.7ex\hbox{E}\kern-.125emX}}
    
\newcommand{\cH}{{\mathcal H}}
\newcommand{\cX}{{\mathcal X}}

\newcommand{\cB}{{\cal B}}
\newcommand{\cS}{{\cal S}}
\newcommand{\Sol}{{\mathcal S_{\cH}}}
    \newcommand{\SolGamma}{{\mathcal S_{\cH|\Gamma}}}
    
    \newcommand{\Ne}{\mathbb{N}}
    \renewcommand{\Re}{\mathbb{R}}
    \newcommand{\Se}{\mathbb{S}}
    \newcommand{\salt}{\vspace*{2ex}}

    \newenvironment{remark}{\begin{remm}\rm }{\hfill \hspace*{1pt} \hfill
    $\triangle$\end{remm}}
    
    \newenvironment{example}{\begin{exx}\rm }{\hfill \hspace*{1pt} \hfill
    $\triangle$
      \end{exx}}
    
    \newenvironment{problem*}[1]{\par\salt \noindent {\bf #1.}}{\hfill \hspace*{1pt} \hfill $\triangle$ \par \salt}

\newcommand\real{\ensuremath{{\mathbb R}}}

\newcommand{\ball}{\mathbb{B}}

\newcommand{\tto}{\ensuremath{\rightrightarrows}}
\newcommand{\dom}{\mathsf{dom}\xspace}
\newcommand{\rge}{\mathsf{rge}\xspace}

\DeclareMathOperator{\interior}{int}

\newcommand{\eps}{\varepsilon}

\graphicspath{{FIGURES/}}

\begin{document}

\title{Reduction Theorems for Hybrid Dynamical Systems}
\author{Manfredi Maggiore, Mario Sassano, Luca Zaccarian 
\thanks{This research was carried out while M. Maggiore
  was on sabbatical leave at the Laboratoire des Signaux et
  Syst\`emes, CentraleSup\'elec, Gif sur Yvette, France. 
M. Maggiore was supported by the Natural Sciences
  and Engineering Research Council of Canada (NSERC).
L. Zaccarian is supported in
  part by grant PowerLyap funded by CaRiTRo.}
\thanks{M. Maggiore is with the Department
  of Electrical and Computer Engineering, University of Toronto,
  Toronto, ON, M5S3G4 Canada \texttt{maggiore@ece.utoronto.ca}}%
\thanks{M. Sassano is with the Dipartimento di
Ingegneria Civile e Ingegneria Informatica, ``Tor Vergata'', Via del
Politecnico 1, 00133, Rome, Italy \texttt{mario.sassano@uniroma2.it}}%
\thanks{L. Zaccarian is with 
LAAS-CNRS, Universit\'e de Toulouse, CNRS, Toulouse, France and with Dipartimento di
    Ingegneria Industriale, University of Trento, Italy \texttt{zaccarian@laas.fr}}%
}

\maketitle

\begin{abstract}
This paper presents reduction theorems for stability, attractivity,
and asymptotic stability of compact subsets of the state space of a
hybrid dynamical system.  Given two closed sets $\Gamma_1
\subset \Gamma_2 \subset \Re^n$, with $\Gamma_1$ compact, the theorems
presented in this paper give conditions under which a qualitative
property of $\Gamma_1$ that holds relative to $\Gamma_2$ (stability,
attractivity, or asymptotic stability) can be guaranteed to also hold
relative to the state space of the hybrid system. As a consequence of
these results, sufficient conditions are presented for the stability
of compact sets in cascade-connected hybrid systems. We also present a
result for hybrid systems with outputs that converge to zero along
solutions. If such a system enjoys a detectability property with
respect to a set $\Gamma_1$, then $\Gamma_1$ is globally
attractive. The theory of this paper is used to develop a hybrid
estimator for the period of oscillation of a sinusoidal signal.
\end{abstract}

\section{Introduction}

\IEEEPARstart{O}{ver} the past ten to fifteen years, research in hybrid dynamical systems theory
has intensified following the work of A.R. Teel and co-authors
(e.g.,~\cite{GoeSanTee09,GoeSanTee12}), which unified previous results
under a common framework, and produced a comprehensive theory of
stability and robustness. In the framework
of~\cite{GoeSanTee09,GoeSanTee12}, a hybrid system is a dynamical
system whose solutions can flow and jump, whereby flows are modelled by
differential inclusions, and jumps are modelled by update
maps. Motivated by the fact that many challenging control
specifications can be cast as problems of set stabilization, the
stability of sets plays a central role in hybrid systems theory.

For continuous nonlinear systems, a useful way to assess whether a
closed subset of the state space is asymptotically stable is to
exploit hierarchical decompositions of the stability problem. To
illustrate this fact, consider the continuous-time cascade-connected
system
\begin{equation}\label{eq:sys}
\begin{aligned}
&\dot x^1 = f_1(x^1,x^2) \\ 
&\dot x^2 = f_2(x^2),
\end{aligned}
\end{equation}
with state $(x^1,x^2) \in \Re^{n_1} \times \Re^{n_2}$, and assume that
$f_1(0,0)=0$, $f_2(0)=0$. To determine whether or not the equilibrium
$(x^1,x^2) = (0,0)$ is asymptotically stable for~\eqref{eq:sys}, one
may equivalently determine whether or not $x^1=0$ is asymptotically
stable for $\dot x^1 = f_1(x^1,0)$ and $x^2=0$ is asymptotically
stable for $\dot x^2 = f_2(x^2)$ (see,
e.g.,~\cite{Vid80,SeiSua90}). This way the stability problem is
decomposed into two simpler sub-problems.

For dynamical systems that do not possess the cascade-connected
structure~\eqref{eq:sys}, the generalization of the
decomposition just described is the focus of the so-called {\em
  reduction problem}, originally formulated by P. Seibert
in~\cite{Sei69,Sei70}.  Consider a dynamical system on $\Re^n$ and two
closed sets $\Gamma_1 \subset \Gamma_2 \subset \Re^n$.  Assume that
$\Gamma_1$ is either stable, attractive, or asymptotically stable
relative to $\Gamma_2$, i.e., when solutions are initialized on
$\Gamma_2$. What additional properties should hold in order that
$\Gamma_1$ be, respectively, stable, attractive, or asymptotically
stable? The global version of this reduction problem is formulated
analogously. 
For continuous dynamical
systems, 
the reduction problem was solved in~\cite{SeiFlo95} for the case when
$\Gamma_1$ is compact, and in~\cite{ElHMag13}  when
$\Gamma_1$ is a closed set. In particular, the work in~\cite{ElHMag13}
linked the reduction problem with a hierarchical control design
viewpoint, in which a hierarchy of control specifications corresponds
to a sequence of sets $\Gamma_1 \subset \cdots \subset \Gamma_l$ to be
stabilized. The design technique of backstepping can be regarded as
one such hierarchical control design problem. Other relevant
literature for the reduction problem is found
in~\cite{Kal99,IggKalOut96}.

In the context of hybrid dynamical systems, the reduction problem is
just as important as its counterpart for continuous nonlinear systems.
To illustrate this fact, we mention three application areas of the theorems
presented in this paper. Additional theoretical implications are
discussed in Section~\ref{sec:reduction_problem}.

Recent literature on stabilization of hybrid limit cycles for bipedal
robots~(e.g.,~\cite{PleGriWesAbb03}) relies on the stabilization of a
set $\Gamma_2$ (the so-called hybrid zero dynamics) on which 
the robot satisfies ``virtual constraints.'' The key idea in this
literature is that, with an appropriate design, one may ensure the
existence of a hybrid limit cycle, $\Gamma_1 \subset \Gamma_2$,
corresponding to stable walking for the dynamics of the robot on the
set $\Gamma_2$. In this context, the theorems presented in this paper
can be used to show that the hybrid limit cycle is asymptotically
stable for the closed-loop hybrid system, without Lyapunov analysis.

Furthermore, as we show in Section~\ref{sec:example}, the problem of
estimating the unknown frequency of a sinusoidal signal can be cast as
a reduction problem involving three sets $\Gamma_1 \subset \Gamma_2
\subset \Gamma_3$. More generally, we envision that the theorems in
this paper may be applied in hybrid estimation problems as already done in \cite{BisoffiAuto17}, whose proof would be simplified by the results of this paper.

Finally, it was shown in~\cite{InvernizziACC18,MichielettoIFAC17,RozMag14} that for underactuated VTOL
vehicles, leveraging reduction theorems one may partition the position
control problem into a hierarchy of two control specifications:
position control for a point-mass, and attitude tracking. Reduction
theorems for hybrid dynamical systems enable employing the
hybrid attitude trackers in~\cite{MayhewTAC11}, allowing one to
generalize the results in~\cite{InvernizziACC18,RozMag14} and obtain global asymptotic
position stabilization and tracking.

{\em Contributions of this paper.} The goal of this paper is to extend
the reduction theorems for continuous dynamical systems found
in~\cite{SeiFlo95,ElHMag13} to the context of hybrid systems modelled
in the framework of~\cite{GoeSanTee09,GoeSanTee12}. We assume
throughout that $\Gamma_1$ is a compact set and develop reduction
theorems for stability of $\Gamma_1$
(Theorem~\ref{thm:reduction_stability}), local/global attractivity of
$\Gamma_1$ (Theorem~\ref{thm:reduction_attractivity}), and
local/global asymptotic stability of $\Gamma_1$
(Theorem~\ref{thm:reduction_asy_stability}). The conditions of the
reduction theorem for asymptotic stability are necessary and
sufficient. Our results generalize the reduction theorems found
  in \cite[Corollary 19]{GoeSanTee09} and \cite[Corollary
    7.24]{GoeSanTee12}, which were used in \cite{Tee10} to develop a
  local hybrid separation principle.

We explore a number of consequences of our reduction theorems. In
Proposition~\ref{prop:stability_cascades} we present a novel result
characterizing the asymptotic stability of compact sets for
cascade-connected hybrid systems. In
Proposition~\ref{prop:detectability} we consider a hybrid system with
an output function, and present conditions guaranteeing that
boundedness of solutions and convergence of the output to zero imply
attractivity of a subset of the zero level set of the output. These
conditions give rise to a notion of detectability for hybrid systems
that had already been investigated in slightly different form
in~\cite{SanGoeTee07}. Finally, in the spirit of the hierarchical
control viewpoint introduced in~\cite{ElHMag13}, we present a
recursive reduction theorem (Theorem~\ref{thm:recursive_reduction}) in
which we consider a chain of closed sets $\Gamma_1 \subset \cdots
\subset \Gamma_l \subset \Re^n$, with $\Gamma_1$ compact, and we
deduce the asymptotic stability of $\Gamma_1$ from the asymptotic
stability of $\Gamma_i$ relative to $\Gamma_{i+1}$ for all $i$.
Finally, the theory developed in this paper is applied to the problem
of estimating the frequency of oscillation of a sinusoidal
signal. Here, the hierarchical viewpoint simplifies an otherwise
difficult problem by decomposing it into three separate sub-problems
involving a chain of sets $\Gamma_1 \subset \Gamma_2 \subset
\Gamma_3$.

{\em Organization of the paper.} In Section~\ref{sec:prelim} we
present the class of hybrid systems considered in this paper and
various notions of stability of sets. The concepts of this section
originate in~\cite{GoeSanTee09,GoeSanTee12,SeiFlo95,ElHMag13}. In
Section~\ref{sec:reduction_problem} we formulate the reduction problem
and its recursive version, and discuss links with the stability of
cascade-connected hybrid systems and the output zeroing problem with
detectability. In Section~\ref{sec:main_results} we present novel
reduction theorems for hybrid systems and their proofs. The results of
Section~\ref{sec:main_results} are employed in
Section~\ref{sec:example} to design an estimator for the frequency of
oscillation of a sinusoidal signal. Finally, in
Section~\ref{sec:conclusion} we make concluding remarks.

{\em Notation.} 
We denote the set of positive real numbers by $\Re_{>0}$, and the set
of nonnegative real numbers by $\Re_{\geq 0}$. We let $\Se^1$ denote
the set of real numbers modulo $2\pi$.  If $x \in \Re^n$, we denote by
$|x|$ the Euclidean norm of $x$, i.e., $|x| = (x^\top x)^{1/2}$. We
denote by $\ball$ the closed unit ball in $\Re^n$, i.e.,
$\ball:=\{x \in \Re^n: |x| \leq 1\}$.  If $\Gamma\subset \Re^n$ and $x
\in \Re^n$, we denote by $|x|_\Gamma$ the point-to-set distance of $x$
to $\Gamma$, i.e., $|x|_\Gamma = \inf_{y \in \Gamma} |x-y|$.  If
$c>0$, we let $B_c(\Gamma) := \{x \in \Re^n : |x|_\Gamma < c\}$, and
$\bar B_c(\Gamma) :=\{x \in \Re^n : |x|_\Gamma \leq c\}$. If $U$ is a
subset of $\Re^n$, we denote by $\bar U$ its closure and by $\interior
U$ its interior. Given two subsets $U$ and $V$ of $\Re^n$, we denote 
their Minkowski sum
by $U + V:= \{ u+v:\; u \in U, v\in V\}$. The empty set is denoted by $\emptyset$.

\section{Preliminary notions}
\label{sec:prelim}

In this paper we use the notion of hybrid system defined
in~\cite{GoeSanTee09,GoeSanTee12} and some notions of set stability
presented in \cite{ElHMag13}. In this section we review the essential
definitions that are required in our development.

Following~\cite{GoeSanTee09,GoeSanTee12}, a hybrid system is a 4-tuple
$\cH = (C,F,D,G)$ satisfying the

\smallskip
\noindent
{\bf Basic Assumptions (\hspace*{-.5ex}\cite{GoeSanTee09,GoeSanTee12})}
\begin{enumerate}[{A}1)]

\item $C$ and $D$ are closed subsets of $\Re^n$.
\item $F: \Re^n \tto \Re^n$ is outer semicontinuous, locally bounded
  on $C$, and such that $F(x)$ is nonempty and convex for each $x \in
  C$.
\item $G :\Re^n \tto \Re^n$ is outer semicontinuous, locally bounded
  on $D$, and such that $G(x)$ is nonempty for each $x \in D$.
\end{enumerate}

A {\em hybrid time domain} is a subset of $\Re_{\geq 0} \times \Ne$
which is the union of infinitely many sets $[t_j,t_{j+1}] \times
\{j\}$, $j\in \Ne$, or of finitely many such sets, with the last one
of the form $ [t_j,t_{j+1}] \times \{j\}$, $[t_j,t_{j+1}) \times
  \{j\}$, or $[t_j,\infty) \times \{j\}$. A {\em hybrid arc} is a
    function $x: \dom(x) \to \Re^n$, where $\dom(x)$ is a hybrid time
    domain, such that for each $j$, the function $t \mapsto x(t,j)$ is
    locally absolutely continuous on the interval $I_j=\{t: (t,j) \in
    \dom(x)\}$. A {\em solution} of $\cH$ is a hybrid arc $x : \dom(x)
    \to \Re^n$ satisfying the following two conditions.

\noindent
{\em Flow condition.} For each $j \in \Ne$ such that $I_j$ has
nonempty interior,
%
%
%
\[
\begin{array}{ll}
\dot x(t,j) \in F(x(t,j)) &\text{ for almost all } t \in I_j, \\
x(t,j) \in C & \text{ for all } t \in [\min I_j,\sup I_j).
\end{array}
\]
%
%
%
{\em Jump condition.} For each $(t,j) \in \dom(x)$ such that $(t,j+1)
\in \dom(x)$,
%
%
%
%
\[
\begin{aligned}
& x(t,j+1) \in G(x(t,j)), \\
& x(t,j) \in D.
\end{aligned}
\]
%
%
%
A solution of $\cH$ is {\em maximal} if it cannot be extended.  
In this paper we will only consider maximal solutions, and
  therefore the adjective ``maximal'' will be omitted in what follows.
%
%
%
%
If $(t_1,j_1), (t_2, j_2) \in \dom(x)$ and $t_1 \leq t_2, j_1 \leq
j_2$, then we write $(t_1,j_1) \preceq (t_2,j_2)$. If at least one
inequality is strict, then we write $(t_1,j_1) \prec (t_2,j_2)$.

A solution $x$ is {\em complete} if $\dom(x)$ is unbounded or,
equivalently, if there exists a sequence $\{ (t_i,j_i)\}_{i \in \Ne}
\subset \dom(x)$ such that $t_i + j_i \to \infty$ as $i \to \infty$.

The set of all maximal solutions of $\cH$  originating from $x_0 \in
\Re^n$ is denoted $\Sol(x_0)$. If $U \subset \Re^n$, then
\[
\Sol(U) :=\bigcup_{x_0 \in U} \Sol(x_0).
\]
We let $\Sol:=\Sol(\Re^n)$. The {\em range} of a hybrid arc $x:\dom(x)
\to \Re^n$ is the set
\[
\rge(x) :=\big\{y \in \Re^n: \big(\exists (t,j) \in \dom(x) \big)
\ y=x(t,j)\big\}.
\]
%
%
%
%
%
%
%
If $U \subset \Re^n$, we define
\[
\rge(\Sol(U)):= \bigcup_{x_0 \in U} \rge\big( \Sol(x_0) \big).
\]

\begin{definition}[Forward invariance]
A set $\Gamma \subset \Re^n$ is {\em strongly forward invariant} for
$\cH$ if
%
%
%
\[
\rge(\Sol(\Gamma)) \subset \Gamma. 
\]
%
%
%
In other words, every solution of $\cH$ starting in $\Gamma$ remains
in $\Gamma$.
%
%
%
%
The set $\Gamma$ is {\em weakly forward invariant} if for
every $x_0 \in \Gamma$ there exists a complete $x \in \Sol(x_0)$ such
that $x(t,j) \in \Gamma$ for all $(t,j) \in \dom(x)$.
\end{definition}

If $\Gamma\subset \Re^n$ is closed, then the {\em restriction} of
$\cH$ to $\Gamma$ is the hybrid system $\cH|_\Gamma :=(C\cap \Gamma,
F,D\cap \Gamma,G)$.  Whenever $\Gamma$ is strongly forward invariant,
solutions that start in $\Gamma$ cannot flow out or jump out of
$\Gamma$. Thus, in this specific case, restricting $\cH$ to $\Gamma$
corresponds to considering only solutions to $\cH$ originating in
$\Gamma$, i.e., $\SolGamma = \Sol(\Gamma)$.

\begin{definition}[stability and attractivity]\label{defn:stability_attractivity}
Let  $\Gamma \subset \Re^n$ be compact. 

\begin{itemize}

\item $\Gamma$ is {\em stable} for $\cH$ if for every $\eps >0$ there
  exists $\delta>0$ such that 
%
%
%
%
\[
\rge(\Sol(B_\delta(\Gamma))) \subset B_\eps(\Gamma).
\]
%
%

\item The {\em basin of attraction} of $\Gamma$ is the largest 
set
of points $p \in \Re^n$ such that each $x \in \Sol(p)$ is bounded and,
if $x$ is complete, then $|x(t,j)|_\Gamma \to 0$ as $t+j \to \infty$,
$(t,j) \in \dom(x)$.

\item $\Gamma$ is {\em attractive} for $\cH$ if the basin of
  attraction of $\Gamma$ contains $\Gamma$ in its interior.

\item $\Gamma$ is {\em globally attractive} for $\cH$ if its 
  basin of attraction is $\Re^n$. 

\item $\Gamma$ is {\em asymptotically stable} for $\cH$ if it is
  stable and attractive, and $\Gamma$ is {\em globally asymptotically
    stable} if it is stable and globally attractive.

\end{itemize}
Let  $\Gamma \subset \Re^n$ be closed.

\begin{itemize}

\item $\Gamma$ is {\em stable} for $\cH$ if for every $\eps>0$ there
  exists an open set $U$ containing $\Gamma$ such that 
\[
\rge(\Sol(U)) \subset B_\eps(\Gamma).
\]
%
%

\item The {\em basin of attraction} of $\Gamma$ is the largest set of
  points $p \in \Re^n$ such that for each $x \in \Sol(p)$,
  $|x|_\Gamma$ is bounded and, if $x$ is complete, then
  $|x(t,j)|_\Gamma \to 0$ as $t+j \to \infty$, $(t,j) \in \dom(x)$.

\item $\Gamma$ is {\em attractive} if the basin of attraction of
  $\Gamma$ contains $\Gamma$ in its interior.

\item $\Gamma$ is {\em globally attractive} if its basin of
  attraction is $\Re^n$.

\item $\Gamma$ is {\em asymptotically stable} if it stable and
  attractive, and {\em globally asymptotically stable} if it is stable
  and globally attractive.
\end{itemize}
\end{definition}

\begin{remark}\label{rem:triviality_of_definitions}
When $C \cup D$ is closed, the properties of stability and
attractivity hold trivially for compact sets $\Gamma$ on which there
are no solutions. More precisely, if $\Gamma \subset \Re^n \setminus
(C \cup D)$, then $\Gamma$ is automatically stable and attractive (and
hence asymptotically stable).  Moreover, all points outside $C \cup D$
trivially belong to its basin of attraction.
\end{remark}

%
%
%
%
\begin{remark}
In~\cite[Definition 7.1]{GoeSanTee12}, the notions of attractivity and
asymptotic stability of compact sets defined above are referred to as
local pre-attractivity and local pre-asymptotic stability. The prefix
``pre'' refers to the fact that the attraction property is only
assumed to hold for complete solutions.  Recent literature on hybrid
systems has dropped this prefix, and in this paper we follow the same
convention.
\end{remark}

\begin{remark}
For the case of closed, non-compact sets,~\cite{GoeSanTee12} adopts
notions of uniform global stability, uniform global pre-attractivity,
and uniform global pre-asymptotic stability (see~\cite[Definition
  3.6]{GoeSanTee12}) that are stronger than the notions presented in
Definition~\ref{defn:stability_attractivity}, but they allow the
authors of~\cite{GoeSanTee12} to give Lyapunov characterizations of
asymptotic stability. In this paper we use weaker definitions to
obtain more general results. Specifically, the results of this paper
whose assumptions concern asymptotic stability of closed sets
(assumptions (ii) and (ii') in Corollary~\ref{cor:asy_stability},
assumptions (i) and (i') in Theorem~\ref{thm:recursive_reduction})
continue to hold when the stronger stability properties
of~\cite{GoeSanTee12} are satisfied.

To illustrate the differences between the above mentioned stability
and attractivity notions for closed sets, in~\cite[Definition
  3.6]{GoeSanTee12} the uniform global stability property requires
that for every $\eps>0$, the open set $U$ of
Definition~\ref{defn:stability_attractivity} be of the form
$B_\delta(\Gamma)$, i.e., a neighborhood of $\Gamma$ of constant
diameter, hence the adjective ``uniform.'' Moreover, \cite[Definition
  3.6]{GoeSanTee12} requires that $\delta \to \infty$ as $\eps \to
\infty$, hence the adjective ``global.'' On the other hand,
Definition~\ref{defn:stability_attractivity} only requires the
existence of a neighborhood $U$ of $\Gamma$, not necessarily of
constant diameter, and without the ``global'' requirement. In
particular, the diameter of $U$ may shrink to zero near points of
$\Gamma$ that are infinitely far from the origin, even as $\eps \to
\infty$. Similarly, the notion of uniform global pre-attractivity
in~\cite[Definition 3.6]{GoeSanTee12} is much stronger than that of
global attractivity in Definition~\ref{defn:stability_attractivity},
for it requires solutions not only to converge to $\Gamma$, but to do
so with a rate of convergence which is uniform over sets of initial
conditions of the form $B_r(\Gamma)$.
\end{remark}

\begin{definition}[local stability and attractivity near a set]
\label{def:local_properties}
Consider two sets $\Gamma_1 \subset \Gamma_2 \subset \Re^n$, and
assume that $\Gamma_1$ is compact.
%
%
%
%
The set $\Gamma_2$ is {\em locally stable near} $\Gamma_1$ for $\cH$
if there exists $r >0$ such that the following property holds. For
every $\eps>0$, there exists $\delta>0$ such that, for each $x \in
\Sol(B_\delta(\Gamma_1))$ and for each $(t,j) \in \dom(x)$, it holds
that if $x(s,k) \in B_r(\Gamma_1)$ for all $(s,k) \in \dom(x)$ with
$(s,k) \preceq (t,j)$, then $x(t,j) \in B_\eps(\Gamma_2)$.  The set
$\Gamma_2$ is {\em locally attractive near $\Gamma_1$} for $\cH$ if
there exists $r>0$ such that $B_r(\Gamma_1)$ is contained in the basin
of attraction of $\Gamma_2$.
\end{definition}

\begin{remark}\label{rem:local_stability}
The notions in Definition~\ref{def:local_properties} originate
  in~\cite{SeiFlo95}.  It is an easy consequence of the definition,
  and it is shown rigorously in the proof of
  Theorem~\ref{thm:reduction_asy_stability}, that local stability of
  $\Gamma_2$ near $\Gamma_1$ is a necessary condition for $\Gamma_1$
  to be stable. In particular, if $\Gamma_1$ is stable, then
  $\Gamma_2$ is locally stable near $\Gamma_1$ for arbitrary
  values\footnote{For this reason, in~\cite{ElHMag13}, local stability
    of $\Gamma_2$ near $\Gamma_1$ is defined by requiring that the
    property holds for any $r>0$.} of $r>0$. Moreover, local
  attractivity of $\Gamma_2$ near $\Gamma_1$ is a necessary condition
  for $\Gamma_1$ to be attractive.  Finally, it is easily seen that if
  $\Gamma_2$ is stable, then $\Gamma_2$ is locally stable near
  $\Gamma_1$, thus local stability of $\Gamma_2$ near $\Gamma_1$ is a
  necessary condition for both the stability of $\Gamma_1$ and the
  stability of $\Gamma_2$.
\end{remark}

\begin{figure}[ht!]
\psfrag{O}{$\Gamma_2$} \psfrag{G}{$\Gamma_1$}
\psfrag{B}{$B_r(\Gamma_1)$} \psfrag{bdg}{$B_\delta(\Gamma_1)$}
\psfrag{beo}{$B_\eps(\Gamma_2)$} \psfrag{1}{$x_1$} \psfrag{2}{$x_2$}
\psfrag{3}{$x_3$} \psfrag{4}{$x_4$}

\centerline{\includegraphics[width=.95\columnwidth]{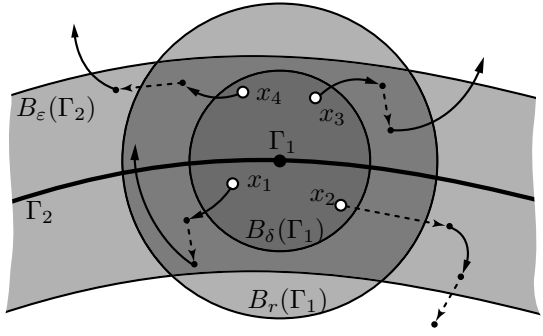}}
\caption{An illustration of local stability of $\Gamma_2$ near
  $\Gamma_1$. Continuous lines denote flow, while dashed lines denote
  jumps. All solutions starting sufficiently close to $\Gamma_1$
  remain close to $\Gamma_2$ so long as they remain in
  $B_r(\Gamma_1)$.  In the figure, the solution from $x_1$ remains in
  $B_r(\Gamma_1)$ and therefore also in $B_\eps(\Gamma_2)$. The
  solution from $x_2$ jumps out of $B_r(\Gamma_1)$, then jumps out of
  $B_\eps(\Gamma_2)$. The solution from $x_3$ flows out of
  $B_r(\Gamma_1)$, then flows out of $B_\eps(\Gamma_2)$. Finally, the
  solution from $x_4$ jumps out of $B_r(\Gamma_1)$, then flows out of
  $B_\eps(\Gamma_2)$.}
\label{fig:local_stab}
\end{figure}

According to Definition~\ref{def:local_properties}, the set $\Gamma_2$
is locally attractive near $\Gamma_1$ if all solutions starting near
$\Gamma_1$ converge to $\Gamma_2$.  Thus $\Gamma_2$ might be locally
attractive near $\Gamma_1$ even when it is not attractive in the sense
of Definition~\ref{defn:stability_attractivity}.  On the other hand,
the set $\Gamma_2$ is locally stable near $\Gamma_1$ if solutions
starting close to $\Gamma_1$ remain close to $\Gamma_2$ so long as
they are not too far from $\Gamma_1$. This notion is illustrated in
Figure~\ref{fig:local_stab}.


\begin{definition}[relative properties]
Consider two closed sets $\Gamma_1 \subset \Gamma_2 \subset \Re^n$.
We say that $\Gamma_1$ is, respectively, stable, (globally)
attractive, or (globally) asymptotically stable {\em relative to}
$\Gamma_2$ if $\Gamma_1$ is stable, (globally) attractive, or
(globally) asymptotically stable for $\cH|_{\Gamma_2}$.
\end{definition}

\begin{example}
To illustrate the definition, consider the linear time-invariant
system
\[
\begin{aligned} 
&\dot x_1 = -x_1 \\
&\dot x_2 = x_2,
\end{aligned}
\]
and the sets $\Gamma_1 =\{(0,0)\}$, $\Gamma_2
=\{(x_1,x_2):x_2=0\}$. Even though $\Gamma_1$ is an unstable
equilibrium, $\Gamma_1$ is globally asymptotically stable relative to
$\Gamma_2$. Now consider the planar system expressed in polar
coordinates $(\rho,\theta) \in \Re_{>0} \times \Se^1$ as
\[
\begin{aligned}
& \dot \theta = \sin^2(\theta/2)+(1-\rho)^2 \\
& \dot \rho=0.
\end{aligned}
\]
Let $\Gamma_1$ be the point on the unit circle $\Gamma_1 =
\{(\theta,\rho): \theta=0, \rho=1\}$, and $\Gamma_2$ be the unit
circle, $\Gamma_2 = \{(\theta,\rho): \rho=1\}$.  On $\Gamma_2$, the
motion is described by $\dot \theta = \sin^2(\theta/2)$. We see that
$\dot \theta \geq 0$, and $\dot \theta =0$ if and only if $\theta = 0$
modulo $2\pi$. Thus $\Gamma_1$ is globally attractive relative to
$\Gamma_2$, even though it is not an attractive equilibrium.
\end{example}

The next two results will be useful in the sequel (see also~\cite[Proposition 3.32]{GoeSanTee12}).

\begin{lemma}\label{lem:restrictions_inherit_stability}
For a hybrid system $\cH:=(C,F,D,G)$, if $\Gamma_1 \subset \Re^n$ is a
closed set which is, respectively, stable, attractive, or globally
attractive for $\cH$, then for any closed set ${\Gamma_2} \subset
\Re^n$, $\Gamma_1$ is, respectively, stable, attractive, or globally
attractive for $\cH|_{\Gamma_2}$.
\end{lemma}

\begin{proof}
The result is a consequence of the fact that each solution of $\cH|_{\Gamma_2}$
is also a solution of $\cH$.
\end{proof}
The next result is a partial converse to
Lemma~\ref{lem:restrictions_inherit_stability}.
\begin{lemma}\label{lem:restrictions_imply_attractivity}
For a hybrid system $\cH:=(C,F,D,G)$,
if $\Gamma_1 \subset {\Gamma_2} \subset \Re^n$ are two closed sets such that
$\Gamma_1$ is compact and $\Gamma_1 \subset \interior {\Gamma_2}$, then:
\begin{enumerate}[(a)]
\item  $\Gamma_1$ is stable for $\cH$ if and only if it is stable for
$\cH|_{\Gamma_2}$.

\item If $\Gamma_1$ is stable for $\cH$, then $\Gamma_1$ is attractive for
  $\cH$ if and only if $\Gamma_1$ is attractive for $\cH|_{\Gamma_2}$.
\end{enumerate}
\end{lemma}

\begin{proof}
{\em Part (a). } By Lemma~\ref{lem:restrictions_inherit_stability}, if
$\Gamma_1$ is stable for $\cH$, then it is also stable for
$\cH|_{\Gamma_2}$. Next assume that $\Gamma_1$ is stable for
$\cH|_{\Gamma_2}$. Since $\Gamma_1$ is compact and contained in the
interior of $\Gamma_2$, there exists $r>0$ such that $B_r(\Gamma_1)
\subset \Gamma_2$. For any $\varepsilon>0$, let
$\varepsilon':=\min(\varepsilon,r)$. By the definition of stability of
$\Gamma_1$, there exists $\delta>0$ such that
\begin{equation}\label{eq:dummy}
\rge(\cS_{\cH|_{\Gamma_2}}(B_\delta(\Gamma_1))) \subset
B_{\varepsilon'}(\Gamma_1).
\end{equation}
Since $B_{\varepsilon'}(\Gamma_1) \subset B_r(\Gamma_1) \subset
\Gamma_2$, we have that solutions of $\cH$ and $\cH|_{\Gamma_2}$
originating in $B_\delta(\Gamma_1)$ coincide, i.e.,
\begin{equation}\label{eq:dummy2}
\cS_{\cH|_{\Gamma_2}}(B_\delta(\Gamma_1)) =
\Sol(B_\delta(\Gamma_1)).
\end{equation}
Substituting~\eqref{eq:dummy2} into~\eqref{eq:dummy} and using the
fact that $\varepsilon'\leq \varepsilon$ we get
\[
\rge(\Sol(B_\delta(\Gamma_1))) \subset
B_{\varepsilon'}(\Gamma_1) \subset B_\varepsilon(\Gamma_1),
\]
which proves that $\Gamma_1$ is stable for $\cH$.

{\em Part (b).}
By Lemma~\ref{lem:restrictions_inherit_stability}, if $\Gamma_1$ is
attractive for $\cH$ then it is also attractive for
$\cH|_{\Gamma_2}$. For the converse, assume that $\Gamma_1$ is
attractive for $\cH|_{\Gamma_2}$. Since $\Gamma_1$ is compact and
contained in the interior of ${\Gamma_2}$, there exists
$\varepsilon>0$ such that $B_\varepsilon(\Gamma_1) \subset
{\Gamma_2}$. Since $\Gamma_1$ is stable for $\cH$, there exists
$\delta>0$ such that
\[
\rge(\Sol(B_\delta(\Gamma_1))) \subset B_\varepsilon(\Gamma_1) \subset {\Gamma_2}.
\]
The above implies that solutions of $\cH$ and $\cH|_{\Gamma_2}$ originating in
$B_\delta(\Gamma_1)$ coincide, i.e.,
\begin{equation}\label{eq:solutions_coincide}
\Sol(B_\delta(\Gamma_1)) = \cS_{\cH|_{\Gamma_2}}(B_\delta(\Gamma_1)).
\end{equation}
Since $\Gamma_1$ is attractive for $\cH|_{\Gamma_2}$, the basin of
attraction of $\Gamma_1$ is a neighborhood of $\Gamma_1$, and
therefore there exists $\delta>0$ small enough to
ensure~\eqref{eq:solutions_coincide} and to ensure that
$B_\delta(\Gamma_1)$ is contained in the basin of
attraction. By~\eqref{eq:solutions_coincide}, $B_\delta(\Gamma_1)$ is
also contained in the basin of attraction of $\Gamma_1$ for system
$\cH$, from which it follows that $\Gamma_1$ is attractive for $\cH$.
\end{proof}

\section{The reduction problem}\label{sec:reduction_problem}

In this section we formulate the reduction problem, discuss its
relevance, and present two theoretical applications: the stability of
compact sets for cascade-connected hybrid systems, and a result
concerning global attractivity of compact sets for hybrid systems with
outputs that converge to zero.

\begin{problem*}{Reduction Problem} Consider a hybrid system $\cH$ satisfying
the Basic Assumptions, and two sets
$\Gamma_1 \subset \Gamma_2 \subset \Re^n$, with $\Gamma_1$ compact and
$\Gamma_2$ closed.  Suppose that $\Gamma_1$ enjoys property $P$
relative to $\Gamma_2$, where $P \in \{$stability, attractivity,
global attractivity, asymptotic stability, global asymptotic
stability$\}$. We seek conditions under which property $P$ holds
relative to $\Re^n$.
\end{problem*}

As mentioned in the introduction, this problem was first formulated by
Paul Seibert in 1969-1970~\cite{Sei69,Sei70}. The solution in the
context of hybrid systems is presented in
Theorems~\ref{thm:reduction_stability},~\ref{thm:reduction_attractivity},~\ref{thm:reduction_asy_stability}
in the next section.

To illustrate the reduction problem, suppose we wish to determine
whether a compact set $\Gamma_1$ is asymptotically stable, and suppose
that $\Gamma_1$ is contained in a closed set $\Gamma_2$, as
illustrated in Figure~\ref{fig:reduction}. In the reduction framework,
the stability question is decomposed into two parts: (1) Determine
whether $\Gamma_1$ is asymptotically stable relative to $\Gamma_2$;
(2) determine whether $\Gamma_2$ satisfies additional suitable
properties (Theorem~\ref{thm:reduction_asy_stability} in
Section~\ref{sec:main_results} states precisely the required
properties). In some cases, these two questions might be easier to
answer than the original one, particularly when $\Gamma_2$ is
  strongly forward invariant, since in this case question (1) would
  typically involve a hybrid system on a state space of lower
  dimension. This sort of decomposition occurs frequently in control
theory, either for convenience or for structural necessity, as we now
illustrate.

\begin{figure}[htb]
\psfrag{1}[c]{$\Gamma_1$} \psfrag{2}{$\Gamma_2$} \psfrag{?}{?}
\centerline{\includegraphics[width=.99\columnwidth]{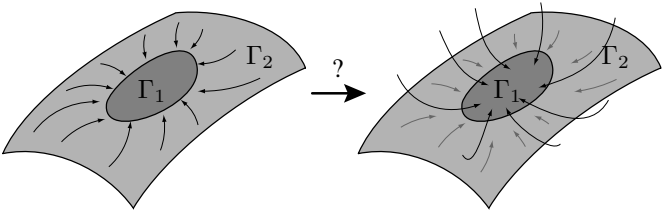}}
\caption{Illustration of the reduction problem when $\Gamma_2$ is
  strongly forward invariant.}
\label{fig:reduction}
\end{figure}

In the context of control systems, the sets $\Gamma_1 \subset
\Gamma_2$ might represent two control specifications organized
hierarchically: the specification associated with set $\Gamma_2$ has
higher priority than that associated with set $\Gamma_1$. Here, the
reduction problem stems from the decomposition of the control design
into two steps: meeting the high-priority specification first, i.e.,
stabilize $\Gamma_2$; then, assuming that the high-priority
specification has been achieved, meet the low-priority specification,
i.e., stabilize $\Gamma_1$ relative to $\Gamma_2$. This point of view
is developed in~\cite{ElHMag13}, and has been applied to the
almost-global stabilization of VTOL vehicles~\cite{RozMag14},
distributed control~\cite{ElHMag13_2,ThuConHu16}, virtual holonomic
constraints~\cite{MagCon13},
robotics~\cite{MohRezMagPet15,OttDieAlb15}, and static or dynamic
allocation of nonlinear redundant actuators~\cite{SassanoEJC16}.
Similar ideas have also been adopted in~\cite{MichielettoIFAC17},
where the concept of local stability near a set, introduced in
Definition~\ref{def:local_properties}, is key to ruling out situations
where the feedback stabilizer may generate solutions that blow up to
infinity.  In the hybrid context, the hierarchical viewpoint described
above has been adopted in~\cite{BisoffiAuto17} to deal with unknown 
jump times in hybrid observation of periodic hybrid exosystems, while
discrete-time results are used in the proof of GAS reported
in~\cite{AlessandriAuto18} for so-called stubborn observers in
discrete time.  In the case of more than two control specifications,
one has the following.

\begin{problem*}{Recursive Reduction Problem}  
Consider a hybrid system $\cH$ satisfying the Basic Assumptions, and
$l$ closed sets $\Gamma_1 \subset \cdots \subset \Gamma_l \subset
\Re^n$, with $\Gamma_1$ compact.  Suppose that $\Gamma_i$ enjoys
property $P$ relative to $\Gamma_{i+1}$ for all $i \in \{1,\ldots,
l\}$, where $P \in \{$stability, attractivity, global attractivity,
asymptotic stability, global asymptotic stability$\}$. We seek
conditions under which the set $\Gamma_1$ enjoys property $P$ relative
to $\Re^n$.

The solution of this problem is found in
Theorem~\ref{thm:recursive_reduction} in the next section.  It is
shown in~\cite{ElHMag13} that the backstepping stabilization technique
can be recast as a recursive reduction problem.
\end{problem*}

As mentioned earlier, the reduction problem may emerge from structural
considerations, such as when the hybrid system is the cascade
interconnection of two subsystems.


\salt
\noindent
{\bf Cascade-connected hybrid systems.}  Consider a hybrid system $\cH
= (C,F$, $D,G)$, where $C = C_1 \times C_2 \subset \Re^{n_1} \times
\Re^{n_2}$, $D = D_1 \times D_2 \subset \Re^{n_1} \times \Re^{n_2}$
are closed sets, and $F: \Re^{n_1+n_2} \tto \Re^{n_1+n_2}$, $G:
\Re^{n_1+n_2} \tto \Re^{n_1+n_2}$ are maps satisfying the Basic
Assumptions. Suppose that $F$ and $G$ have the upper triangular
structure
%
%
%
%
%
%
\begin{equation}\label{eq:upper_triangular}
F(x^1,x^2) = \begin{bmatrix} F_1(x^1,x^2) \\ F_2(x^2)
\end{bmatrix}, \ G(x^1,x^2) = \begin{bmatrix} G_1(x^1,x^2) \\ G_2(x^2)
\end{bmatrix},
\end{equation}
where $(x^1,x^2) \in \Re^{n_1} \times \Re^{n_2}$.  Define $\hat F_1
:\Re^{n_1} \tto \Re^{n_1}$ and $\hat G_1 : \Re^{n_1} \tto \Re^{n_1}$
as
\begin{equation}\label{eq:upper_triangular2}
\hat F_1(x^1):=F_1(x^1,0), \ \hat G_1(x^1):=G_1(x^1,0).
\end{equation}
With these definitions, we can view $\cH$ as the cascade connection of
the hybrid systems
\[
\cH_1= (C_1,\hat F_1,D_1,\hat G_1), \ \cH_2 = (C_2,F_2,D_2,G_2),
\]
with $\cH_2$ driving $\cH_1$.  The following result is a corollary of
Theorem~\ref{thm:reduction_asy_stability} in
Section~\ref{sec:main_results}.  It generalizes to the hybrid setting
classical results for continuous time-invariant dynamical systems in,
e.g.,~\cite{Vid80,SeiSua90}. Using
Theorems~\ref{thm:reduction_stability}
and~\ref{thm:reduction_attractivity}, one may formulate analogous
results for the properties of attractivity and stability.
\begin{proposition}\label{prop:stability_cascades}
Consider the hybrid system $\cH:=(C_1 \times C_2, F, D_1\times
D_2,G)$, with maps $F, G$ given in~\eqref{eq:upper_triangular}, and
the two hybrid subsystems $\cH_1 := (C_1,\hat F_1, D_1,\hat G_1)$ and
$\cH_2:=(C_2,F_2,D_2,G_2)$ satisfying the Basic Assumptions, with maps
$\hat F_1,\hat G_1$ given in~\eqref{eq:upper_triangular2}. Let $\hat
\Gamma_1 \subset \Re^{n_1}$ be a compact set, and denote
\begin{equation}\label{eq:cascade:Gamma1}
\Gamma_1 = \{(x^1,x^2) \in \Re^{n_1} \times \Re^{n_2} : x^1 \in \hat
\Gamma_1, \, x^2=0\}.
\end{equation}
Suppose that $0 \in C_2 \cup D_2$. Then the following holds:

\begin{enumerate}[(i)]

\item $\Gamma_1$ is asymptotically stable for $\cH$ if $\hat \Gamma_1$
  is asymptotically stable for $\cH_1$ and $0 \in \Re^{n_2}$ is
  asymptotically stable for $\cH_2$.

\item $\Gamma_1$ is globally asymptotically stable for $\cH$ if $\hat
  \Gamma_1$ is globally asymptotically stable for $\cH_1$, $0 \in
  \Re^{n_2}$ is globally asymptotically stable for $\cH_2$, and all
  solutions of $\cH$ are bounded.

\end{enumerate}
\end{proposition}
 
The result above is obtained from
Theorem~\ref{thm:reduction_asy_stability} in
Section~\ref{sec:main_results} setting $\Gamma_1$ as
in~\eqref{eq:cascade:Gamma1}, and $\Gamma_2 = \{(x_1,x_2) \in
\Re^{n_1} \times \Re^{n_2} : x_2=0\}$. The restriction
$\cH|_{\Gamma_2}$ is given by
\[
\left.{\mathcal H}\right|_{\Gamma_2}= \left(C_1 \times
\{0\}, \begin{bmatrix} F_1(x_1,0) \\ F_2(0) \end{bmatrix}, D_1 \times
\{0\}, \begin{bmatrix} G_1(x_1,0) \\ G_2(0) \end{bmatrix} \right),
\]
from which it is straightforward to see that $\Gamma_1$ is (globally)
asymptotically stable relative to $\Gamma_2$ if and only if $\hat
\Gamma_1$ is (globally) asymptotically stable for $\cH_1$. It is also
clear that if $0 \in \Re^{n_2}$ is (globally) asymptotically stable
for $\cH_2$, then $\Gamma_2$ is (globally) asymptotically stable for
$\cH$. 
The converse, however, is not true. Namely, the (global)
asymptotic stability of $\Gamma_2$ for $\cH$ does not imply that $0
\in \Re^{n_2}$ is (globally) asymptotically stable for $\cH_2$, which
is why Proposition~\ref{prop:stability_cascades} states only
sufficient conditions.  The reason is that the set of hybrid arcs
$x_2(t,j)$ generated by solutions of $\cH$ is generally smaller than
the set of solutions of $\cH_2$. This phenomenon is illustrated in the
next example.

\begin{example}\label{example:cascade}
Consider the cascade connected system $\cH=(C_1 \times C_2, F, D_1
\times D_2, G)$, with $C_1 = \{1\}$, $C_2 = \Re$, $D_1=D_2
=\emptyset$, and
\[
F(x_1,x_2) = \begin{bmatrix} 1 \\ x_2
\end{bmatrix}.
\]
All solutions of $\cH$ have the form $(1,x_2(0,0))$, and are defined
only at $(t,j)=(0,0)$.  Since the origin $(x_1,x_2)=(0,0)$ is not
contained in $C \cup D$, it is trivially asymptotically stable for
$\cH$ (see Remark~\ref{rem:triviality_of_definitions}). Moreover,
there are no complete solutions, and all solutions are constant, hence
bounded, which implies that the basin of attraction of the origin is
the entire $\Re^2$. Hence the origin is globally asymptotically stable
for $\cH$.  On the other hand, $\cH_2$ is the linear time-invariant 
continuous-time system on
$\Re$ with dynamics $\dot x_2 = x_2$, clearly unstable. This example
shows that the condition, in
Proposition~\ref{prop:stability_cascades}, that $0$ be (globally)
asymptotically stable for $\cH_2$ is not necessary.
\end{example}

Proposition \ref{prop:stability_cascades} is to be compared to
  \cite[Theorem 1]{Tee10}, where the author presents an analogous
  result for a different kind of cascaded hybrid system. The notion of
  cascaded hybrid system used in
  Proposition~\ref{prop:stability_cascades} is one in which a jump is
  possible only if the states $x^1$ and $x^2$ are simultaneously in
  their respective jump sets, $D_1$ and $D_2$, and a jump event
  involves both states, simultaneously. On the other hand, the notion
  of cascaded hybrid system proposed in~\cite{Tee10} is one in which
  jumps of $x^1$ and $x^2$ occur independently of one another, so that
  when $x^1$  jumps nontrivially, $x^2$ remains constant, and vice versa. Moreover,
  in~\cite{Tee10} the jump and flow sets are not expressed as
  Cartesian products of sets in the state spaces of the two
  subsystems.

Another circumstance in which the reduction problem plays a prominent
role is the notion of detectability for systems with outputs.

\salt
\noindent
{\bf Output zeroing with detectability.}  
%
%
Consider a hybrid system $\cH$ satisfying the Basic Assumptions, with
a continuous output function $h: \Re^n \to \Re^k$, and let $\Gamma_1$
be a compact, strongly forward invariant subset of $h^{-1}(0)$.
Assume that all solutions on $\Gamma_1$ are complete.  Suppose that
all $x \in \Sol$ are bounded. Under what circumstances does the
property $h(x(t,j)) \to 0$ for all complete $x \in \Sol$ imply that
$\Gamma_1$ is globally attractive?  This question arises in the
context of passivity-based stabilization of
equilibria~\cite{ByrIsiWil91} and closed sets~\cite{ElHMag10} for
continuous control systems. In the hybrid systems setting, a similar
question arises when using virtual constraints to stabilize hybrid
limit cycles for biped robots
(e.g.,~\cite{PleGriWesAbb03,WesGriKod03,WesGriCheChoMor07}). In this
case the zero level set of the output function is the virtual
constraint.

Let $\Gamma_2$ denote the maximal weakly forward invariant subset
contained in $h^{-1}(0)$.
%
%
%
%
Using the sequential compactness of the space of solutions of
$\cH$~\cite[Theorem 4.4]{GoeTee06}, one can show that the closure of a
weakly forward invariant set is weakly forward invariant.  This fact
and the maximality of $\Gamma_2$ imply that $\Gamma_2$ is closed.
Furthermore, since $\Gamma_1$ is strongly forward invariant, contained
in $h^{-1}(0)$, and all solutions on it are complete, necessarily
$\Gamma_1 \subset \Gamma_2$. It turns out (see the proof of
Proposition~\ref{prop:detectability} below) that any bounded complete
solution $x$ such that $h(x(t,j)) \to 0$ converges to $\Gamma_2$.

In light of the discussion above, the question we asked earlier can be
recast as a reduction problem: Suppose that $\Gamma_2$ is globally
attractive. What stability properties should $\Gamma_1$ satisfy
relative to $\Gamma_2$ in order to ensure that $\Gamma_1$ is globally
attractive for $\cH$?  The answer, provided by
Theorem~\ref{thm:reduction_attractivity} in
Section~\ref{sec:main_results}, is that $\Gamma_1$ should be globally
asymptotically stable relative to $\Gamma_2$ (attractivity is not
enough, as shown in Example~\ref{ex:attractivity:counterexample}
below).

Following\footnote{In~\cite{SanGoeTee07}, the authors adopt a
  different definition of detectability, one that requires $\Gamma_1$
  to be globally attractive, instead of globally asymptotically
  stable, relative to $\Gamma_2$. When they employ this property,
  however, they make the extra assumption that $\Gamma_1$ be stable
  relative to $\Gamma_2$.}~\cite{ElHMag10}, the hybrid system $\cH$ is
said to be {\em $\Gamma_1$-detectable from $h$} if $\Gamma_1$ is
globally asymptotically stable relative to $\Gamma_2$, where
$\Gamma_2$ is the maximal weakly forward invariant subset contained in
$h^{-1}(0)$.

Using the reduction theorem for attractivity in
Section~\ref{sec:main_results}
(Theorem~\ref{thm:reduction_attractivity}), we get the answer to the
foregoing output zeroing question.

\begin{proposition}\label{prop:detectability}
Let $\cH$ be a hybrid system satisfying the Basic Assumptions, $h :
\Re^n \to \Re^k$ a continuous function, and $\Gamma_1 \subset
h^{-1}(0)$ be a compact set which is strongly forward invariant for
$\cH$, such that all solutions from $\Gamma_1$ are complete.  If 1)
$\cH$ is $\Gamma_1$-detectable from $h$, 2) each $x \in \Sol$ is
bounded, and 3) all complete $x \in \Sol$ are such that $h(x(t,j)) \to
0$, then $\Gamma_1$ is globally attractive. 
\end{proposition}

\begin{proof}
%
%
%
%
Let $\Gamma_2$ be the maximal weakly forward invariant subset of
$h^{-1}(0)$.  This set is closed by sequential compactness of the
space of solutions of $\cH$~\cite[Theorem 4.4]{GoeTee06}. By
assumption, any $x \in \Sol$ is bounded.  If $x \in \Sol$ is complete,
by~\cite[Lemma~3.3]{SanGoeTee07}, the positive limit set $\Omega(x)$
is nonempty, compact, and weakly invariant. Moreover, $\Omega(x)$ is
the smallest closed set approached by $x$. Since $h(x(t,j)) \to 0$ and
$h$ is continuous, $\Omega(x) \subset h^{-1}(0)$. Since $\Omega(x)$ is
weakly forward invariant and contained in $h^{-1}(0)$, necessarily
$\Omega(x) \subset \Gamma_2$. Thus $\Gamma_2$ is globally attractive
for $\cH$.  Since $\Gamma_1$ is strongly forward invariant, contained
in $h^{-1}(0)$, and on it all solutions are complete, $\Gamma_1$ is
contained in $\Gamma_2$, the maximal set with these properties.  By
the $\Gamma_1$-detectability assumption, $\Gamma_1$ is globally
asymptotically stable relative to $\Gamma_2$. By
Theorem~\ref{thm:reduction_attractivity}, we conclude that $\Gamma_1$
is globally attractive.
\end{proof}

\section{Main results}\label{sec:main_results}
  
In this section we solve the reduction problem, presenting reduction
theorems for stability, (global) attractivity, and (global) asymptotic
stability. We also present the solution of the recursive reduction
problem for the property of asymptotic stability.

\begin{theorem}[Reduction theorem for 
stability]\label{thm:reduction_stability} For a hybrid system $\cH$
  satisfying the Basic Assumptions, consider two sets $\Gamma_1
  \subset \Gamma_2 \subset \Re^n$, with $\Gamma_1$ compact and
  $\Gamma_2$ closed.
 If
\begin{enumerate}[(i)]
\item $\Gamma_1$ is asymptotically stable relative to $\Gamma_2$,

\item $\Gamma_2$ is locally stable near $\Gamma_1$,
\end{enumerate}
then $\Gamma_1$ is stable for $\cH$.
\end{theorem}

\begin{remark}
As argued in Remark~\ref{rem:local_stability}, local stability of
$\Gamma_2$ near $\Gamma_1$ (assumption (ii)) is a necessary condition
in Theorem~\ref{thm:reduction_stability}.  In place of this
assumption, one may use the stronger assumption that $\Gamma_2$ be
stable, which might be easier to check in practice but is not a
necessary condition (see for example system \eqref{eq:forremark} in Example~\ref{ex:forremark}).
There are situations, however, when the local stability property is
essential and emerges quite naturally from the context of the
problem. This occurs, for instance, when solutions far from $\Gamma_1$
but near $\Gamma_2$ have finite escape times. For examples of such
situations,  refer to~\cite{GreMasMag17}
and~\cite{MichielettoIFAC17}.
\end{remark}

\begin{proof}
Hypotheses (i) and (ii) imply that there exists a scalar $r>0$
such that:
%
\begin{itemize}
\item[(a)] Set $\Gamma_1$ is globally asymptotically stable for system
  $\mathcal{H}_{r,0} := (C \cap \Gamma_2 \cap \bar B_r(\Gamma_1), F, D
  \cap \Gamma_2 \cap \bar B_r(\Gamma_1),G)$,

\item[(b)] 
Given system $\mathcal{H}_r := \cH|_{\bar B_r(\Gamma_1)}$
for each $\varepsilon>0$,
$\exists \delta>0$ such that all solutions to $\mathcal{H}_r$ satisfy:
$$
|x(0,0)|_{\Gamma_1} \leq \delta \; \Rightarrow \; |x(t,j)|_{\Gamma_2} \leq \varepsilon, \; \forall (t,j) \in \dom(x).
$$

\end{itemize}
Since $\Gamma_1$ is contained in the interior of $\bar B_r(\Gamma_1)$, 
by Lemma~\ref{lem:restrictions_imply_attractivity} to prove stability
of $\Gamma_1$ for ${\mathcal H}$ it suffices to prove stability of
$\Gamma_1$ for system $\mathcal{H}_r$ introduced in (b).
  The rest of the proof follows similar steps to the
proof of stability reported in \cite[Corollary 7.24]{GoeSanTee12}.

From item (a) and due to \cite[Theorem 7.12]{GoeSanTee12}, there exists a
class $\mathcal{KL}$ bound $\beta \in \mathcal{KL}$ and, due to
\cite[Lemma 7.20]{GoeSanTee12} applied with a constant perturbation
function $x \mapsto \rho(x) =\bar \rho$ and with ${\mathcal U}=\real^n$, for each
$\varepsilon >0$ there exists $\bar \rho>0$ such that defining
\begin{equation}
\label{eq:rhobarsets}
\begin{array}{rcl}
C_{\bar \rho,r} &:=& C \cap  \bar B_{\bar \rho}(\Gamma_2)\cap
  \bar B_r(\Gamma_1) \\
  &\subset& 
\{x \in \Re^n:\; (x+  \bar \rho \ball) \cap (C \cap \Gamma_2 \cap  \bar B_r(\Gamma_1)) \neq \emptyset \} \\
D_{\bar \rho,r} &:=& D \cap  \bar B_{\bar \rho}(\Gamma_2)\cap
  \bar B_r(\Gamma_1)  \\
  &\subset& 
\{x \in \Re^n:\; (x+  \bar \rho \ball) \cap (D \cap \Gamma_2 \cap \bar B_r(\Gamma_1)) \neq \emptyset \}
\end{array}
\end{equation}
and introducing system ${\mathcal H}_{\bar \rho,r} := (C_{\bar \rho,r},F ,D_{\bar \rho,r} ,G)$, we have\footnote{Note that 
for a constant perturbation $\rho(x) = \bar \rho$ the inflated flow
and jump sets in \cite[Definition~6.27]{GoeSanTee12} are exactly $\bar
\rho$ inflations of the original ones.}
\begin{equation}
\begin{array}{l}
  |x(t,j)|_{\Gamma_1} \leq \beta( |x(0,0)|_{\Gamma_1}, t+j) +\frac{\varepsilon}{2}, \\
\hspace{3cm} \forall (t,j) \in \dom(x), 
\forall x \in {\mathcal S}_{{\mathcal H}_{\bar \rho,r}}
\end{array}
\label{eq:KLbound}
\end{equation}

Let $\varepsilon >0$ be given. Let $\bar \rho>0$ be such that (\ref{eq:KLbound}) holds. Due to item (b) above, there exists a small enough $\delta >0$ such that $\beta(\delta, 0) \leq \frac{\varepsilon}{2}$ and
\begin{equation}
(x \in {\mathcal S}_{\mathcal{H}_r}, |x(0,0)|_{\Gamma_1} \! \leq  \delta) \Rightarrow |x(t,j)|_{\Gamma_2}\! \leq \bar \rho, \forall (t,j)\! \in\! \dom(x) .
\label{eq:rhobar_sols}
\end{equation}
Then
the solutions considered in (\ref{eq:rhobar_sols}) are also solutions
of ${\mathcal H}_{\bar \rho,r}$ because they remain in $B_{\bar
  \rho}(\Gamma_2)$. Since these are solutions of ${\mathcal H}_{\bar
  \rho,r}$, we may apply (\ref{eq:KLbound}) to get
\begin{equation}
|x(t,j)|_{\Gamma_1} \leq \beta( \delta, 0) +\frac{\varepsilon}{2}\leq 
 \frac{\varepsilon}{2} +\frac{\varepsilon}{2} = \varepsilon, \;\forall (t,j) \in \dom(x) ,
\end{equation}
which completes the proof.
\end{proof}

\begin{example}
\label{ex:forremark}
Assumption (i) in the above theorem cannot be replaced by the weaker
requirement that $\Gamma_1$ be stable relative to $\Gamma_2$.  To
illustrate this fact, consider the linear time-invariant system
\[
\begin{aligned}
& \dot x_1 =x_2 \\
& \dot x_2 =0,
\end{aligned}
\]
with $\Gamma_1 = \{(0,0)\}$ and $\Gamma_2 = \{(x_1,x_2):
x_2=0\}$. Although $\Gamma_1$ is stable relative to $\Gamma_2$ and
$\Gamma_2$ is stable, $\Gamma_1$ is an unstable equilibrium. On the
other hand, consider the system
\[
\begin{aligned}
& \dot x_1 =-x_1 +x_2 \\
& \dot x_2 =0,
\end{aligned}
\]
with the same definitions of $\Gamma_1$ and $\Gamma_2$. Now $\Gamma_1$
is asymptotically stable relative to $\Gamma_2$, and $\Gamma_2$ is
stable. As predicted by Theorem~\ref{thm:reduction_stability},
$\Gamma_1$ is a stable equilibrium. Finally, let $\sigma: \Re \to [0,
  1]$ be a $C^1$ function such that $\sigma(s)=0$ for $|s|\leq 1$ and
$\sigma(s) =1$ for $|s| \geq 2$, and consider the system
\begin{align}
\begin{array}{l}
 \dot x_1 =-x_1 (1-\sigma(x_1))+x_2^2 \\[.1cm]
 \dot x_2 =\sigma(x_1)x_2,
\end{array}
\label{eq:forremark}
\end{align}
with the earlier definitions of $\Gamma_1$ and $\Gamma_2$. One can see
that $\Gamma_1$ is asymptotically stable relative to $\Gamma_2$, and
$\Gamma_2$ is unstable. For the former property, note that the motion
on $\Gamma_2$ is described by $\dot x_1 = -x_1(1-\sigma(x_1))$, a
$C^1$ differential equation which near $\{x_1=0\}$ reduces to $\dot
x_1=-x_1$.  To see that $\Gamma_2$ is an unstable set, note that if
$x_1(0) \geq 2$, then $x_1(t) \geq x_1(0)$ and $\dot x_2 =
x_2$. Namely, solutions move away from $\Gamma_2$. On the other hand,
$\Gamma_2$ is locally stable near $\Gamma_1$, because as long as
$|x_1|\leq 1$, $\dot x_2 = 0$.  By
Theorem~\ref{thm:reduction_stability}, $\Gamma_1$ is a stable
equilibrium.
\end{example}

\begin{theorem}[Reduction theorem for
 attractivity]\label{thm:reduction_attractivity} For a hybrid system
  $\cH$ satisfying the Basic Assumptions, consider two sets $\Gamma_1
  \subset \Gamma_2 \subset \Re^n$, with $\Gamma_1$ compact and
  $\Gamma_2$ closed.  Assume that

\begin{enumerate}[(i)]
\item $\Gamma_1$ is globally asymptotically stable relative to
  $\Gamma_2$,

\item $\Gamma_2$ is globally attractive,
\end{enumerate}
then the basin of attraction of $\Gamma_1$ is the set
\begin{equation}
  \cB := \{x_0
\in \Re^n : \text{all } x \in \Sol(x_0) \text{ are bounded}\}.
\label{eq:Bset}
\end{equation}
In particular, if $\cB$ contains $\Gamma_1$ in its
interior, then $\Gamma_1$ is attractive. If all solutions of $\cH$ are
bounded, then $\Gamma_1$ is globally attractive.
\end{theorem}

\begin{proof}
By definition, any bounded non complete solution belongs to the basin
of attraction of $\Gamma_1$.  The proof amounts then to showing that
any bounded and complete solution $x \in \Sol$ converges to
$\Gamma_1$, so that all points in $\cB$ defined in \eqref{eq:Bset} are
contained in its basin of attraction.  Conversely, any solution in the
basin of attraction of $\Gamma_1$ is bounded by definition, so it
belongs to $\cB$.
Hypothesis (i) corresponds to the following fact:
\begin{itemize}
\item[(a)] Set $\Gamma_1$ is globally asymptotically stable for system 
$\cH|_{\Gamma_2} := (C \cap \Gamma_2, F, D \cap \Gamma_2,G)$.
\end{itemize}

The rest of the proof follows similar steps to the proof of
attractivity reported in \cite[Corollary~7.24]{GoeSanTee12}.  Given
any bounded and complete solution $x\in \Sol$, define $M :=
\max\nolimits\limits_{(t,j) \in \dom(x)} |x(t,j)|_{\Gamma_1}$.
Convergence of $x$ to $\Gamma_1$ is established by showing that for
each $\varepsilon$, there exists $T\geq 0$ such that
\begin{equation}
|x(t,j)|_{\Gamma_1} \leq \varepsilon, \quad \forall (t,j) \in \dom (x): t+j \geq T.
\label{eq:conv_x}
\end{equation}

From item (a) above, and applying \cite[Theorem~7.12]{GoeSanTee12},
there exists a uniform class $\mathcal{KL}$ bound $\beta \in
\mathcal{KL}$ on the solutions to system $\cH|_{\Gamma_2}$.  Fix an
arbitrary $\varepsilon>0$.  To establish (\ref{eq:conv_x}), due to
\cite[Lemma~7.20]{GoeSanTee12} applied to $\cH|_{\Gamma_2}$ with
$\mathcal{U}= \real^n$, with a constant perturbation function $x
\mapsto \rho(x) =\bar \rho$ and with the compact set $K = \bar
B_M(\Gamma_1)$ (to be used in the definition of semiglobal practical
$\mathcal{KL}$ asymptotic stability of \cite[Definition~7.18]{GoeSanTee12}),
there exists a small enough $\bar \rho>0$ such that
defining~\footnote{Note that the set inclusions in
  \eqref{eq:rhobarsets_attr} always hold for a small enough $\bar
  \rho$. Indeed, even in the peculiar case when $C \cap \Gamma_2$ is
  empty, since $C$ and $\Gamma_2$ are closed, it is possible to pick
  $\bar \rho$ small enough so that $C \cap \bar B_{\bar
    \rho}(\Gamma_2)$ is empty too, and then the inclusions
  \eqref{eq:rhobarsets_attr} hold because both sides are empty
  sets. Similar arguments apply when $D \cap \Gamma_2$ is empty. }
\begin{equation}
\label{eq:rhobarsets_attr}
\begin{array}{rcl}
C_{\bar \rho} \!\!&:=&\!\! \bar
  B_M(\Gamma_1) \; \cap\; C \cap \bar B_{\bar \rho}(\Gamma_2) \\
  &\subset& \!\!  
 \bar B_M(\Gamma_1) \cap
 \{x \in \Re^n:\; (x+  \bar \rho \ball) \cap ((C \cap \Gamma_2 ) \neq \emptyset \}      \\
D_{\bar \rho}\!\! &:=& \!\! \bar
  B_M(\Gamma_1) \; \cap\; D \cap \bar B_{\bar \rho}(\Gamma_2) \\
  &\subset&   \!\!
  \bar B_M(\Gamma_1) \cap
\{x \in \Re^n:\; (x+  \bar \rho \ball) \cap ((D \cap \Gamma_2 ) \neq \emptyset \}  
\end{array}
\end{equation}
and introducing system 
${\mathcal H}_{\bar \rho} := (C_{\bar \rho},F ,D_{\bar \rho} ,G)$,
we have 
\begin{align}
\label{eq:KLbound_attr}
|\bar x(t,j)|_{\Gamma_1} &\leq \beta( |\bar x(0,0)|_{\Gamma_1}, t+j) +\frac{\varepsilon}{2}, \\
\nonumber
    &\leq  \beta( M, t+j) +\frac{\varepsilon}{2}, \;\; \forall (t,j) \in \dom(\bar x), \forall \bar x \in {\mathcal S}_{{\mathcal H}_{\bar \rho}}.
\end{align}
Define now $T_2>0$ satisfying $\beta( M,T_2 ) \leq \frac{\varepsilon}{2}$, and obtain:
\begin{equation}
\label{eq:conv_bound}
\bar x \in {\mathcal S}_{{\mathcal H}_{\bar \rho}} \quad \Rightarrow \quad 
|\bar x(t,j)|_{\Gamma_1} \leq \varepsilon, \forall (t,j) \in \dom(\bar x): t+j \geq T_2.
\end{equation}
Moreover, from hypothesis (ii), there exists $T_1>0$ such that
$|x(t,j)|_{\Gamma_2} \leq \bar \rho$ for all $(t,j) \in \dom (x)$ satisfying $t+j \geq T_1$. As a consequence, the tail of solution $x$ (after $t+j \geq T_1$) is a solution to ${\mathcal H}_{\bar \rho}$. By virtue of (\ref{eq:conv_bound}), equation (\ref{eq:conv_x}) is established with $T= T_1+T_2$ and the proof is completed.
\end{proof}

\begin{example} \label{ex:circles}
Consider a hybrid system with continuous states $x=(x_1,x_2,x_3)
\in \Re^3$ and a discrete state $q \in \{1,-1\}$. The 
dynamics are defined as
\[
\begin{aligned}
& \dot x_1 = q x_2      &\quad&x_1^+= x_1\\
& \dot x_2 = -q x_1     &&x_2^+=x_2 \\
& \dot x_3 = x_1^2 -x_3 &&x_3^+=x_3/2 \\
& \dot q =0,            &&q^+ =-q , 
\end{aligned}
\]
and the flow and jump sets are selected as closed sets ensuring that along
flowing solutions
we have $x_1(t,j) > 0 \Rightarrow q(t,j) = 1$ and 
$x_1(t,j) < 0 \Rightarrow q(t,j) = -1$. To this end, when the
solution hits the set $\{x_1=0\}$, the discrete state is toggled, $q^+
= -q$, and the state $x_3$ is halved, $x_3^+ = x_3/2$.  
In particular, we select
\[
\begin{aligned}
& C= \{(x,q): x_1 \geq 0, q=1\} \cup \{(x,q): x_1
  \leq 0, q =-1\}, \\
& D=\{(x,q): x_1 =0, q=1\} \cup \{(x,q): x_1 = 0, q =-1\}, \\
\end{aligned}
\]
For any flowing solution starting in $C$, the states $(x_1,x_2)$
describe an arc of a circle centered at $(x_1,x_2)=(0,0)$. The
direction of motion 
is clockwise on the half-space $x_1> 0$, and counter-clockwise on $x_1
<0$. Each solution reaches the set $\{(x,q): x_1=0\}$ in finite
time. On this set, the only complete solutions are Zeno, namely, the
discrete state $q$ persistently toggles. The set
\[
\Gamma_2:=\{(x,q): x_1=0\}
\]
is, therefore, globally attractive for $\cH$. It is, however,
unstable, as solutions of the $(x_1,x_2)$-subsystem starting
arbitrarily close to $\Gamma_2$ with $x_2>0$ evolve along arcs of
circles that move away from $\Gamma_2$.  On $\Gamma_2$, the flow is
described by the differential equation $\dot x_3 = -x_3$, while the
jumps are described by the difference equation $x_3^+ = x_3 /2$. Thus
the $x_2$ axis
\[
\Gamma_1 :=\{(x,q) \in \Gamma_2 : x_3=0\},
\]
%
%
%
is globally asymptotically stable relative to $\Gamma_2$.  Since the states
$(x_1,x_2)$ are bounded, so is the $x_3$ state.  By
Theorem~\ref{thm:reduction_attractivity}, $\Gamma_1$ is globally
attractive for $\cH$. On the other hand, $\Gamma_1$ is unstable for $\cH$.
\end{example}

\begin{example}\label{ex:attractivity:counterexample}
In
Theorem~\ref{thm:reduction_attractivity}, one may not replace
assumption (i) by the weaker requirement that $\Gamma_1$ be attractive
relative to $\Gamma_2$. We illustrate this fact with an example taken
from~\cite{ElH11}.  Consider the smooth differential equation
\[
\begin{aligned}
&\dot{x}_1=(x_2^2+x_3^2)(-x_2)\\
&\dot{x}_2=(x_2^2+x_3^2)(x_1)\\
&\dot{x}_3=-x_3^3,
\end{aligned}
\]
and the sets $\Gamma_1=\{(x_1,x_2,x_3):x_2=x_3=0\}$ and
$\Gamma_2=\{(x_1,x_2,x_3):x_3=0\}$. One can see that $\Gamma_2$ is
globally asymptotically stable, and the motion on $\Gamma_2$ is
described by the system
\[
\begin{aligned}
&\dot{x}_1=-x_2 (x_2^2)\\ &\dot{x}_2=x_1 (x_2^2).
\end{aligned}
\]
On $\Gamma_1\subset \Gamma_2$, every point is an equilibrium.  Phase
curves on $\Gamma_2$ off of $\Gamma_1$ are concentric semicircles
$\{x_1^2+x_2^2=c\}$, and therefore $\Gamma_1$ is a global, but
unstable, attractor relative to $\Gamma_2$. As shown in
Figure~\ref{fig:ce}, for initial conditions not in $\Gamma_2$ the
trajectories are bounded and their positive limit set is a circle
inside $\Gamma_2$ which intersects $\Gamma_1$ at equilibrium
points. Thus $\Gamma_1$ is not attractive.
\begin{figure}[htb]
\psfrag{x1}{{\small $x_1$}} \psfrag{x2}{{\small $x_2$}}
\psfrag{x3}{{\small $x_3$}} \psfrag{G}{{\small $\Gamma_1$}}
\psfrag{O}{{\small $\Gamma_2$}}
\centerline{\includegraphics[width=.9\columnwidth]{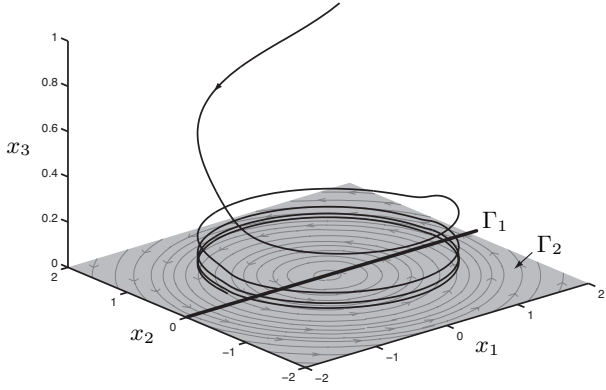}}
\caption{Example~\ref{ex:attractivity:counterexample}: $\Gamma_1$ is
  globally attractive relative to $\Gamma_2$, $\Gamma_2$ is globally
  asymptotically stable, and yet $\Gamma_1$ is not attractive.}
\label{fig:ce}
\end{figure}
\end{example}

\begin{theorem}[Reduction theorem for asymptotic stability]
\label{thm:reduction_asy_stability} 
For a hybrid system $\cH$ satisfying the Basic Assumptions, consider
two sets $\Gamma_1 \subset \Gamma_2 \subset \Re^n$, with $\Gamma_1$
compact and $\Gamma_2$ closed.  Then $\Gamma_1$ is asymptotically
stable if, and only if
\begin{enumerate}[(i)]
\item $\Gamma_1$ is asymptotically stable relative to $\Gamma_2$,

\item $\Gamma_2$ is locally stable near $\Gamma_1$,

\item $\Gamma_2$ is locally attractive near $\Gamma_1$.
\end{enumerate}
Moreover, $\Gamma_1$ is globally asymptotically stable for $\cH$ if,
and only if,
\begin{enumerate}[(i')]
\item $\Gamma_1$ is globally asymptotically stable relative to
  $\Gamma_2$,

\item $\Gamma_2$ is locally stable near $\Gamma_1$,

\item $\Gamma_2$ is globally attractive,

\item all solutions of $\cH$ are bounded.
\end{enumerate}
\end{theorem}

\begin{proof}
$(\Leftarrow)$ We begin by proving the local version of the theorem.

By assumption (i), there exists $r>0$ such that $\Gamma_1$ is globally
asymptotically stable relative to the set $\Gamma_{2,r} :=\Gamma_2
\cap \bar B_r(\Gamma_1)$ for $\cH$. By
Lemma~\ref{lem:restrictions_inherit_stability}, the same property
holds for the restriction $\cH_r:=\cH|_{\bar B_r(\Gamma_1)}$.

By assumption (iii) and by making, if necessary, $r$ smaller,
$\Gamma_{2,r}$ is globally attractive for $\cH_r$.

By Theorem~\ref{thm:reduction_attractivity}, the basin of attraction
of $\Gamma_1$ for $\cH_r$ is the set of initial conditions from which
solutions of $\cH_r$ are bounded. Since the flow and jump sets of
$\cH_r$ are compact, all solutions of $\cH_r$ are bounded, and thus
$\Gamma_1$ is attractive for $\cH_r$.

Assumptions (i) and (ii) and Theorem~\ref{thm:reduction_stability}
imply that $\Gamma_1$ is stable for $\cH$. Since $\Gamma_1$ is
  contained in the interior of $\bar B_r(\Gamma_1)$, by
Lemma~\ref{lem:restrictions_imply_attractivity} the attractivity of
$\Gamma_1$ for $\cH_r$ implies the attractivity of $\Gamma_1$ for
$\cH$. Thus $\Gamma_1$ is asymptotically stable for $\cH$.

For the global version, it suffices to notice that assumptions (i'),
(iii'), and (iv') imply, by Theorem~\ref{thm:reduction_attractivity},
that $\Gamma_1$ is globally attractive for $\cH$.

$(\Rightarrow)$ Suppose that $\Gamma_1$ is asymptotically
stable. By Lemma~\ref{lem:restrictions_inherit_stability}, $\Gamma_1$
is asymptotically stable for 
$\cH|_{\Gamma_2}$, and thus condition (i) holds.
By~\cite[Proposition~6.4]{GoeTee06}, the basin of attraction of
$\Gamma_1$ is an open set $\cB$ containing $\Gamma_1$, each solution
$x \in \Sol(\cB)$ is bounded and, if it is complete, it converges to
$\Gamma_1$. Since $\Gamma_1 \subset \Gamma_2$, such a solution
converges to $\Gamma_2$ as well. Thus the basin of attraction of
$\Gamma_2$ contains $\cB$, proving that $\Gamma_2$ is locally
attractive near $\Gamma_1$ and condition (iii) holds. To prove that
$\Gamma_2$ is locally stable near $\Gamma_1$, let $r>0$ and $\eps>0$
be arbitrary. Since $\Gamma_1$ is stable, there exists $\delta>0$ such
that each $x \in \Sol(B_\delta(\Gamma_1))$ remains in
$B_\eps(\Gamma_1)$ for all hybrid times in its hybrid time
domain. Since $\Gamma_1 \subset \Gamma_2$, $B_\eps(\Gamma_1) \subset
B_\eps(\Gamma_2)$. Thus each $x \in \Sol(B_\delta(\Gamma_1))$ remains
in $B_\eps(\Gamma_2)$ for all hybrid times in its hybrid time
domain. In particular, it also does so for all the hybrid times for
which it remains in $B_r(\Gamma_1)$. This proves that condition (ii)
holds.

Suppose that $\Gamma_1$ is globally asymptotically stable. The proof
that conditions (i'), (ii'), (iii') hold is a straightforward
adaptation of the arguments presented above. Since $\Gamma_1$ is
globally attractive, its basin of attraction is $\Re^n$. Since
$\Gamma_1$ is compact, by definition all solutions originating in its
basin of attraction are bounded. Thus condition (iv') holds.
\end{proof}

Theorems~\ref{thm:reduction_stability}
and~\ref{thm:reduction_asy_stability} generalize to the hybrid setting
analogous results for continuous systems in~\cite{SeiFlo95,ElHMag13,Son89b}.
The following corollary is of particular interest.

\begin{corollary}\label{cor:asy_stability}
For a hybrid system $\cH$ satisfying the Basic Assumptions, consider
two sets $\Gamma_1 \subset \Gamma_2 \subset \Re^n$, with $\Gamma_1$
compact and $\Gamma_2$ closed. If
\begin{enumerate}[(i)]
\item $\Gamma_1$ is asymptotically stable relative to $\Gamma_2$,

\item $\Gamma_2$ is asymptotically stable,

\end{enumerate}
then $\Gamma_1$ is asymptotically stable.  Moreover, if
\begin{enumerate}[(i')]
\item $\Gamma_1$ is globally asymptotically stable relative to $\Gamma_2$,

\item $\Gamma_2$ is globally asymptotically stable,

\end{enumerate}
then $\Gamma_1$ is asymptotically stable with basin of attraction
given by the set of initial conditions from which all solutions are
bounded. In particular, if all solutions are bounded, then $\Gamma_1$
is globally asymptotically stable.
\end{corollary}

\begin{proof}
If $\Gamma_2$ is asymptotically stable then $\Gamma_2$ is locally
attractive near $\Gamma_1$. Moreover, for each $\eps>0$ there exists
an open set $U$ containing $\Gamma_2$ such that each $x \in \Sol(U)$
remains in $B_\eps(\Gamma_2)$ for all hybrid times in its hybrid time
domain. Since $\Gamma_1 \subset \Gamma_2$, $\Gamma_1$ is contained in
$U$. Since $\Gamma_1$ is compact, there exists $\delta>0$ such that
$B_\delta(\Gamma_1) \subset U$. Thus each solution $x \in
\Sol(B_\delta(\Gamma_1))$ remains in $B_\eps(\Gamma_2)$ for all hybrid
times in its hybrid time domain, implying that $\Gamma_2$ is locally
stable near $\Gamma_1$. By Theorem~\ref{thm:reduction_asy_stability},
$\Gamma_1$ is asymptotically stable.  An analogous argument holds for
the global version of the corollary.
\end{proof}

If in
Theorems~\ref{thm:reduction_stability},~\ref{thm:reduction_attractivity},
and~\ref{thm:reduction_asy_stability} one replaces $\Re^n$ by a closed
subset $\cX$ of $\Re^n$, then the conclusions of the theorems hold
relative to $\cX$, for one can apply the theorems to the restriction
$\cH|_{\cX}$. This allows one to apply the theorems inductively to
finite sequences of nested subsets $\Gamma_1 \subset \cdots \subset
\Gamma_l$ to solve the recursive reduction problem.

\begin{theorem}[Recursive reduction theorem for asymptotic 
stability] \label{thm:recursive_reduction} For a hybrid system $\cH$
  satisfying the Basic Assumptions, consider $l$ sets $\Gamma_1
  \subset \cdots \subset \Gamma_l \subset \Gamma_{l+1}:=\Re^n$, with
  $\Gamma_1$ compact and all $\Gamma_i$ closed. If
\begin{enumerate}[(i)]
\item $\Gamma_i$ is asymptotically stable relative to $\Gamma_{i+1}$,
  $i=1,\ldots,l$, 
\end{enumerate}
then $\Gamma_1$ is asymptotically stable for $\cH$. On the other hand, if 
\begin{enumerate}[(i')]

\item $\Gamma_i$ is globally asymptotically stable relative to $\Gamma_{i+1}$,
  $i=1,\ldots,l$,

\item all $x \in \Sol$ are bounded,
\end{enumerate}
then $\Gamma_1$ is globally asymptotically stable for $\cH$.
\end{theorem}

Analogous statements hold, {\em mutatis mutandis}, for the properties
of stability and attractivity
(see~\cite[Proposition~14]{ElHMag13}). The proof of the theorem above
is contained in that of~\cite[Proposition~14]{ElHMag13} and is
therefore omitted.

\section{Adaptive hybrid observer for uncertain internal models}
\label{sec:example}

Consider a LTI system described by equations of the form
\begin{subequations}\label{eq:exosystem_nominal}
\begin{align}
\dot \chi &= \left[\begin{array}{cc}  0 & -\omega \\ \omega & 0 \end{array} \right]\chi := S \chi, \\
y &= \left[ \begin{array}{cc} 1 & 0 \end{array} \right]\chi := H \chi,
\end{align}
\end{subequations}
\noindent with $\omega \in \mathbb{R}$ not precisely known, for which however lower and upper
bounds are assumed to be available, namely $\omega_m < \omega < \omega_M$, 
$\omega_m, \omega_M \in \mathbb{R}_{+}$. Note that (\ref{eq:exosystem_nominal}) can be considered a hybrid
system with empty jump set and jump map. Suppose in addition
that the norm of the initial condition $\chi(0,0)$ is upper and lower
bounded, namely $\chi_{m} \leq |\chi(0,0)| \leq \chi_{M}$,
for some known positive constants $\chi_{m}$ and $\chi_{M}$. By the nature of the
dynamics in (\ref{eq:exosystem_nominal}), the bounds above imply the existence
of a compact set $\mathcal{W} := \{\chi \in \mathbb{R}^2 : |\chi| \in [\chi_m, \chi_M]\}$ that is 
strongly forward invariant for (\ref{eq:exosystem_nominal}) and where solutions to (\ref{eq:exosystem_nominal})
are constrained to evolve.

The objective of this section consists in 
estimating the period of oscillation, namely $2\pi/\omega$ with $\omega$ unknown, and in (asymptotically) 
reconstructing the
state of the system (\ref{eq:exosystem_nominal}) via the measured output $y$. It is shown that
this task can be reformulated in terms of the results discussed in the previous sections.
Towards this end, let 
\begin{equation}\label{eq:hyb_estimator}
\left\{ \!\!\! \begin{array}{ccl}
\dot{\hat{\chi}} \!\!\! &=& \!\!\! \hat{S}(T)\hat{\chi}+\hat{L}(T)(y-H\hat{\chi}), \\ 
\dot q  \!\!\! &=& \!\!\! 0, \\
\dot T   \!\!\! &=& \!\!\!  0, \\
\dot \tau   \!\!\! &=& \!\!\!  1,
\end{array}
\right. \!\!
\left\{ \!\!\!  \begin{array}{ccl}
\hat{\chi}^{+}  \!\!\!\! &=& \!\!\! \hat{\chi}, \\
q^{+}  \!\!\!\! &=& \!\!\! {\rm sign}(y), \\
T^{+}  \!\!\!\! &=& \!\!\! \lambda T + (1-\lambda)2\tau, \\
\tau^{+}  \!\!\!\! &=& \!\!\! 0,
\end{array} \right.
\end{equation}
\noindent with $\lambda \in [0,1)$, denote the \emph{flow} and \emph{jump} maps, respectively, 
of the proposed hybrid estimator, where the matrices $\hat{S}(T)$ and $\hat{L}(T)$ are defined as
\begin{equation}\label{eq:hatS}
\hat{S}(T) := \left[ \begin{array}{cc} 0 & -\dfrac{2 \pi}{T} \\ \dfrac{2 \pi}{T} & 0 \end{array} \right], \hspace{0.5cm}
\hat{L}(T) := \left[ \begin{array}{c} \dfrac{4 \pi}{T} \\[12pt] 0 \end{array} \right]\,,
\end{equation}
\noindent which are such that $(\hat{S}(T)-\hat{L}(T)H)$ is Hurwitz. Note that
the lower bound $T_m$ on $T$ specified below guarantees that matrix $\hat{S}(T)$ is well-defined.

Intuitively, the rationale behind the definition of flow and jump sets for the hybrid
estimator given below is that the system is forced to jump whenever the sign of the
logic variable $q$ is different from the sign of the output $y$.
Therefore, homogeneity of the dynamics implies that 
$\tau$ is eventually upper-bounded by some value $\bar{\tau} = \pi/\omega_m$. Moreover, 
note
that the lower and upper bounds on $\omega$ induce
similar bounds on the possible values of $T$, namely $2 \pi/\omega_M = T_{m} < T < T_{M} = 2 \pi/\omega_m$.
Denoting by $\Xi$ the space where state $\xi:= (\chi,\hat{\chi},q,T,\tau)$ evolves,
\begin{equation*}
  \Xi:=\mathcal{W} \times \mathbb{R}^{2} \times \{-1,1\} \times [T_m, T_M] \times [0, \pi/\omega_m],
\end{equation*}
the closed-loop system (\ref{eq:exosystem_nominal})-(\ref{eq:hyb_estimator}) is then completed by the flow set
\begin{equation}\label{eq:flow_set}
\mathcal{C} := \{ (\chi,\hat{\chi},q,T,\tau) \in \Xi: qy \geq -\sigma \}\,,
\end{equation}
\noindent and by the jump set
\begin{equation}\label{eq:jump_set}
\mathcal{D} := \{ (\chi, \hat{\chi},q,T,\tau) \in \Xi: |y|\geq \sigma, qy \leq -\sigma \}
\end{equation}
\noindent for some $\sigma > 0$ that should be selected
smaller than $\chi_m$ to guarantee that the output trajectory, under the assumptions for the
initial conditions of (\ref{eq:exosystem_nominal}), intersects the line $qy = -\sigma$.
Note that $\mathcal{C}$ and $\mathcal{D}$ depend only on the output $y$.

\begin{figure}[t]
\centering
  \includegraphics[width=0.95\columnwidth]{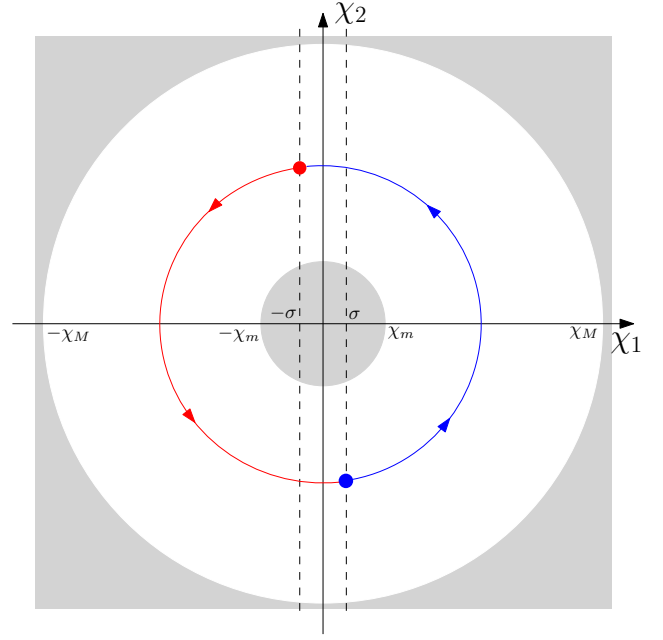}
  \caption{The white \emph{doughnut} represents the set $\mathcal{W}$. The
  red/blue curve is a solution $\chi(t,j)$ where the dots represents jump instants. The solution is blue
  in regions where $\mathfrak{h}(\chi(t,j)) = -1$ and is red in regions where $\mathfrak{h}(\chi(t,j)) = 1$}\label{fig:Trajectories}
\end{figure}

Adopting the notation introduced in the previous sections, 
define the functions $\mathfrak{h} : \mathbb{R}^{2} \rightarrow \{-1,1\}$
as
\begin{equation}\label{eq:h_frak}
\mathfrak{h}(\chi) := \left\{  
\begin{array}{rcr}
-1, & {\rm if} & \chi_1 \geq \sigma \hspace{0.2cm} \lor \hspace{0.2cm} (|\chi_1|<\sigma \wedge \chi_2 > 0)  \\
1, & {\rm if} & \chi_1 \leq -\sigma \hspace{0.2cm} \lor \hspace{0.2cm} (|\chi_1|<\sigma \wedge \chi_2 < 0)
\end{array}
\right.
\end{equation}
\noindent and $\varrho : \mathbb{R}^2 \times \mathbb{R} \rightarrow \mathbb{R}$
as $\varrho(\chi,\tau) := H e^{S\left(\pi/\omega -\tau \right)}\chi - \mathfrak{h}(\chi)\sigma$, which is constant along
flowing solutions because
\begin{equation}\label{eq:rho_const}
\dot{\varrho}(\chi,\tau) = -H e^{S\left(\pi/\omega -\tau \right)} S\chi + H e^{S\left(\pi/\omega-\tau \right)} \dot{\chi} = 0\,,
\end{equation}
\noindent which is zero if and only if $\tau$ is suitably synchronized with $\chi$, namely such that 
$\tau^+ = \pi/\omega$: this would in turn guarantee that $T^{+} = 2 \pi/ \omega$ at the next jump provided that also $T = 2 \pi/ \omega$.
Then,
consider the sets
\begin{equation}\label{eq:Gamma_3_IM}
\begin{split}
\Gamma_3 := \Big \{ \xi & \in \Xi: \varrho(\chi,\tau)=0 \Big \}\,,
\end{split}
\end{equation}
\begin{equation}\label{eq:Gamma_2_IM}
\begin{split}
\Gamma_2 := \Big \{ \xi & \in \Gamma_3: T = \dfrac{2 \pi}{\omega}\Big \}
\end{split}
\end{equation}
\noindent and
\begin{equation}\label{eq:Gamma_1_IM}
\begin{split}
\Gamma_1 := \Big \{ \xi & \in \Gamma_2: \chi = \hat{\chi} \Big \}
\end{split}
\end{equation}
\noindent with $\xi := (\chi,\hat{\chi},q,T,\tau)$, which clearly satisfy $\Gamma_1 \subset \Gamma_2 \subset \Gamma_3$.
Roughly speaking, on the set $\Gamma_1$ the state $\hat{\chi}$ of the
hybrid estimator (\ref{eq:hyb_estimator}) is perfectly synchronized
with that of system (\ref{eq:exosystem_nominal}), $\Gamma_2$ consists of the set of states that ensure $T^{+} = 2\pi/\omega$
at the next jump, while $\Gamma_3$ prescribes the correct value of the
initial timer $\tau$, depending on the initial phase of $\chi$, such that
at jumps $\tau$ coincides with $\pi/\omega$. Note that $\Gamma_1$ is compact, by the hypothesis
on $\mathcal{W}$, while $\Gamma_2$ and $\Gamma_3$ are closed. 

Let us now show GAS of $\Gamma_1$ by using reductions theorems. To this end,
we apply the recursive version of Theorem~\ref{thm:reduction_asy_stability} given in Theorem~\ref{thm:recursive_reduction}. In particular, we show
GAS of $\Gamma_1$ relative to $\Gamma_2$, GAS of $\Gamma_2$ relative to $\Gamma_3$,
GAS of $\Gamma_3$ and finally boundedness of solutions.
To begin with, it can be shown
that $\Gamma_1$ is globally asymptotically stable relative to $\Gamma_2$. 
In fact, letting $\eta_1 = \chi - \hat{\chi}$ denote
the estimation error, then its dynamics restricted to $\Gamma_2$, due to the trivial jumps
of $\chi$ and $\hat{\chi}$, is described by the hybrid system defined by the flow dynamics
\begin{equation}\label{eq:e_dot}
\dot \eta_1 = S\chi - \hat{S}(T)\hat{\chi} - \hat{L}(T)H\eta_1 = (S-\hat{L}(T)H)\eta_1 \,,
\end{equation}
\noindent which is obtained by considering that, on the set
$\Gamma_2$, $\hat{S}(T) = S$, for $\xi \in \mathcal{C}$, and the jump
dynamics $\eta_1^+ = \eta_1$ for $\xi \in \mathcal{D}$.  The claim
follows by recalling that $\hat{L}(T)$ is such that $(S-\hat{L}(T)H)$
is Hurwitz and by \emph{persistent flowing} conditions of stability
\cite[Proposition~3.27]{GoeSanTee12}.

\begin{figure}[t]
\centering
  \includegraphics[width=1.1\columnwidth]{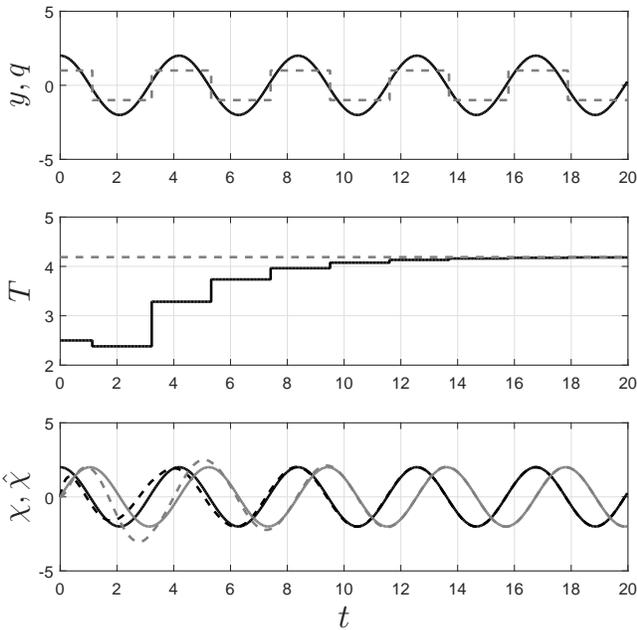}
  \caption{Top Graph: time histories of the function $y$ generated by 
    (\ref{eq:exosystem_nominal}) and of the state $q(t,k)$, solid and dashed lines,
      respectively. Middle Graph: time histories of the estimate 
  $T(t,k)$, converging to the correct value of the period of oscillation $2 \pi/ \omega$. Bottom Graph: time histories of $\hat{\chi}_{1}(t,k)$ (dark) and $\hat{\chi}_{2}(t,k)$ (gray), solid lines,
  converging to the actual states $\chi_{1}(t,k)$ and $\chi_{2}(t,k)$, dashed lines.  }\label{fig:Estimates}
\end{figure}

Moreover, $\Gamma_2$ is globally asymptotically stable relative to $\Gamma_3$.
To show this, let $\eta_2 = T - 2\pi/\omega$ and recall that
all the trajectories of (\ref{eq:hyb_estimator}) that remain in $\Gamma_3$
are characterized by the property that $\tau = \pi/\omega$ at the time of jump.
Therefore, the dynamics of $\eta_2$ restricted to $\Gamma_3$ 
is described by the hybrid system defined by the flow dynamics $\dot{\eta}_2 = 0$,
for $\xi \in \mathcal{C}$ and the jump dynamics
\begin{equation}\label{eq:eta2_plus}
\eta_2^{+} = T^{+} - \dfrac{2\pi}{\omega} = \lambda\left(T-\dfrac{2\pi}{\omega} \right) = \lambda \eta_2 \,,
\end{equation}
\noindent for $\xi \in \mathcal{D}$. Asymptotic stability of
$\Gamma_2$ relative to $\Gamma_3$ then follows by \emph{persistent
  jumping} stability conditions \cite[Proposition~3.24]{GoeSanTee12},
which applies because $\sigma > \chi_m$, and by recalling that $0 \leq
\lambda < 1$. In addition, global attractivity of $\Gamma_3$ can be
shown by relying on the fact that $\tau(t_2,1)$, namely the value of
$\tau$ before the second jump, is equal to $\pi/\omega$, hence
implying that $\varrho(\chi(t,k),\tau(t,k)) = 0$ for $(t,k) \in {\rm
  dom} \, \varrho$ with $k>1$. Stability of $\Gamma_3$, on the other
hand, follows by noting that a perturbation $\delta$ on $\tau(0,0)$
with respect to the values in $\Gamma_3$, $i.e.$ values that satisfy
$\varrho(\chi,\tau) = 0$, results in $\tau(t_1,0) = \pi/\omega +
\varepsilon(\delta)$, with $\varepsilon$ a class-$\mathcal{K}$
function of $\delta$.

Finally, boundedness of the trajectories of the state $\chi$ and of $q$, $T$ and $\tau$
follows by the existence of the strongly forward invariant set $\mathcal{W}$ - described by 
the lower, $\chi_m$, and upper, $\chi_M$, bounds - and by definition of the flow and jump
sets, respectively. Therefore, to conclude global asymptotic stability of the set $\Gamma_1$
it only remains to show that the trajectories of $\hat{\chi}$ are bounded. Towards this end,
recall the flow dynamics of $\hat{\chi}$ in (\ref{eq:hyb_estimator}), namely
\begin{equation}\label{eq:dyn_hatchi}
\dot{\hat{\chi}} = (S(T)-\hat{L}(T)C_o)\hat{\chi} + \hat{L}(T)C_o\chi := M(T)\hat{\chi} + \hat{L}(T)C_o\chi\,,
\end{equation}
\noindent with $M(T)$, and its derivative with respect to $T$, uniformly bounded in $T$, since
$T \in [T_m, T_M]$, and Hurwitz uniformly in $T$ by definition of $\hat{L}(T)$, whereas
the jump dynamics is described by $\hat{\chi}^{+} = \hat{\chi}$. Thus, by applying 
\cite[Lemma~5.12]{khalil_book}, it follows that there exists a unique positive definite solution
$P(T)$ to the Lyapunov equation $P(T)M(T)+M(T)^{\top}P(T) = -I$, with the additional property
that $c_1 |\hat{\chi}|^2 \leq \hat{\chi}^{\top}P(T)\hat{\chi} \leq c_2 |\hat{\chi}|^2$, for some
positive constants $c_1$ and $c_2$.
Boundedness of the trajectories of $\hat{\chi}$ then follows by standard manipulations
on the time derivative of the functions $\hat{\chi}^{\top}P(T)\hat{\chi}$ along
the trajectories of (\ref{eq:dyn_hatchi}) and by noting that $\hat{L}(T)$ is uniformly
bounded, by the definition of $\hat{L}$ and of $T$, and by recalling
that $|\chi|$ is uniformly bounded by definition of the strongly forward invariant compact set $\mathcal{W}$.

In the following numerical simulations, we suppose that $\omega = 1.5$ and
we let $\sigma = 0.25$ and $\lambda = 0.5$. Moreover, 
we let $\chi(0,0) = [2,\,0]'$ and $\hat{\chi}(0,0) = [0,\,0]'$, while the remaining components of the
estimator are initialized as $q(0,0) = 1$, $T(0,0) = 2.5$ and $\tau(0,0) = 0$.
The top graph of Figure~\ref{fig:Estimates} depicts the time histories of the function 
$y$ generated by (\ref{eq:exosystem_nominal}) and of the state $q(t,k)$, solid and dashed lines,
respectively. The middle graph of Figure~\ref{fig:Estimates} shows the time histories of the estimate 
$T(t,k)$, converging to the correct value of the period of oscillation $2 \pi/ \omega$, while the bottom
graph displays the time histories of $\hat{\chi}_{1}(t,k)$ (dark) and $\hat{\chi}_{2}(t,k)$ (gray), solid lines,
converging to the actual states $\chi_{1}(t,k)$ and $\chi_{2}(t,k)$, dashed lines.

\section{Conclusion} \label{sec:conclusion}
In this paper we presented three reduction theorems for stability,
local/global attractivity, and local/global asymptotic stability of
compact sets for hybrid dynamical systems, along with a number of
their consequences. The proofs of these results rely crucially on the
${\cal KL}$ characterization of robustness of asymptotic stability of
compact sets found in~\cite[Theorem~7.12]{GoeSanTee12}. A different proof
technique is possible which generalizes the proofs found
in~\cite{ElHMag13}. As a future research direction, we conjecture
that, similarly to what was done in~\cite{ElHMag13} for continuous
dynamical systems, it may be possible to state reduction theorems for
hybrid systems in which the set $\Gamma_1$ is only assumed to be
closed, not necessarily bounded.

In addition to the applications listed in the introduction, the
  reduction theorems presented in this paper may be employed to
  generalize the position control laws for VTOL vehicles presented
  in~\cite{RozMag14,MichielettoIFAC17}, by replacing continuous
  attitude stabilizers with hybrid ones, such as the one found
  in~\cite{MayhewTAC11}.  Furthermore, the results of this paper may
  be used to generalize the allocation techniques
  of~\cite{SassanoEJC16}, possibly following similar ideas to those
  in~\cite{SerraniAuto15}.

\section*{Acknowledgments} 
 The authors wish to thank Andy Teel for fruitful discussions and
 Antonio Lor\'ia for making their research collaboration possible.

\bibliographystyle{IEEEtranS}
\bibliography{IEEEabrv,biblio}

\begin{thebibliography}{10}
\providecommand{\url}[1]{#1}
\csname url@samestyle\endcsname
\providecommand{\newblock}{\relax}
\providecommand{\bibinfo}[2]{#2}
\providecommand{\BIBentrySTDinterwordspacing}{\spaceskip=0pt\relax}
\providecommand{\BIBentryALTinterwordstretchfactor}{4}
\providecommand{\BIBentryALTinterwordspacing}{\spaceskip=\fontdimen2\font plus
\BIBentryALTinterwordstretchfactor\fontdimen3\font minus
  \fontdimen4\font\relax}
\providecommand{\BIBforeignlanguage}[2]{{%
\expandafter\ifx\csname l@#1\endcsname\relax
\typeout{** WARNING: IEEEtranS.bst: No hyphenation pattern has been}%
\typeout{** loaded for the language `#1'. Using the pattern for}%
\typeout{** the default language instead.}%
\else
\language=\csname l@#1\endcsname
\fi
#2}}
\providecommand{\BIBdecl}{\relax}
\BIBdecl

\bibitem{AlessandriAuto18}
A.~Alessandri and L.~Zaccarian, ``Stubborn state observers for linear
  time-invariant systems,'' \emph{Automatica}, vol.~88, pp. 1--9, Feb. 2018.

\bibitem{BisoffiAuto17}
A.~Bisoffi, L.~Zaccarian, M.~D. Lio, D.~Carnevale, and J.~Contributors,
  ``Hybrid cancellation of ripple disturbances arising in {AC/DC} converters,''
  \emph{Automatica.}, vol.~77, pp. 344--352, 2017.

\bibitem{ByrIsiWil91}
C.~Byrnes, A.~Isidori, and J.~Willems, ``Passivity, feedback equivalence, and
  the global stabilization of nonlinear systems,'' \emph{IEEE Transactions on
  Automatic Control}, vol.~36, pp. 1228--1240, 1991.

\bibitem{ElH11}
M.~El-Hawwary, ``Passivity methods for the stabilization of closed sets in
  nonlinear control systems,'' Ph.D. dissertation, University of Toronto, 2011.

\bibitem{ElHMag10}
M.~El-Hawwary and M.~Maggiore, ``Reduction principles and the stabilization of
  closed sets for passive systems,'' \emph{IEEE Transactions on Automatic
  Control}, vol.~55, no.~4, pp. 982--987, 2010.

\bibitem{ElHMag13_2}
------, ``Distributed circular formation stabilization for dynamic unicycles,''
  \emph{IEEE Transactions on Automatic Control}, vol.~58, no.~1, pp. 149--162,
  2013.

\bibitem{ElHMag13}
------, ``Reduction theorems for stability of closed sets with application to
  backstepping control design,'' \emph{Automatica}, vol.~49, no.~1, pp.
  214--222, 2013.

\bibitem{SerraniAuto15}
S.~Galeani, A.~Serrani, G.~Varano, and L.~Zaccarian, ``On input
  allocation-based regulation for linear over-actuated systems,''
  \emph{Automatica}, vol.~52, pp. 346--354, 2015.

\bibitem{GoeSanTee09}
R.~Goebel, R.~G. Sanfelice, and A.~Teel, ``Hybrid dynamical systems,''
  \emph{IEEE Control Systems}, vol.~29, no.~2, pp. 28--93, 2009.

\bibitem{GoeSanTee12}
R.~Goebel, R.~Sanfelice, and A.~Teel, \emph{Hybrid Dynamical Systems: modeling,
  stability, and robustness}.\hskip 1em plus 0.5em minus 0.4em\relax Princeton
  University Press, 2012.

\bibitem{GoeTee06}
R.~Goebel and A.~Teel, ``Solutions to hybrid inclusions via set and graphical
  convergence with stability theory applications,'' \emph{Automatica}, vol.~42,
  no.~4, pp. 573--587, 2006.

\bibitem{GreMasMag17}
L.~Greco, P.~Mason, and M.~Maggiore, ``Circular path following for the
  spherical pendulum on a cart,'' in \emph{{IFAC} {W}orld {C}ongress},
  Toulouse, France, July 2017.

\bibitem{IggKalOut96}
A.~Iggidr, B.~Kalitin, and R.~Outbib, ``Semidefinite {L}yapunov functions
  stability and stabilization,'' \emph{Mathematics of Control, Signals and
  Systems}, vol.~9, pp. 95--106, 1996.

\bibitem{InvernizziACC18}
D.~Invernizzi, M.~Lovera, and L.~Zaccarian, ``Geometric tracking control of
  underactuated {VTOL} {UAV}s,'' in \emph{American Control Conference},
  Milwaukee (WI), USA, Jul. 2018, pp. 3609--3614.

\bibitem{Kal99}
B.~S. Kalitin, ``B-stability and the {F}lorio-{S}eibert problem,''
  \emph{Differential Equations}, vol.~35, pp. 453--463, 1999.

\bibitem{khalil_book}
H.~Khalil, \emph{Nonlinear Systems}, 2nd~ed.\hskip 1em plus 0.5em minus
  0.4em\relax USA: Prentice Hall, 1996.

\bibitem{MagCon13}
M.~Maggiore and L.~Consolini, ``Virtual holonomic constraints for
  {E}uler--{L}agrange systems,'' \emph{IEEE Transactions on Automatic Control},
  vol.~58, no.~4, pp. 1001--1008, 2013.

\bibitem{MayhewTAC11}
C.~Mayhew, R.~Sanfelice, and A.~Teel, ``Quaternion-based hybrid control for
  robust global attitude tracking,'' \emph{IEEE Transactions on Automatic
  Control}, vol.~56, no.~11, pp. 2555--2566, 2011.

\bibitem{MichielettoIFAC17}
G.~Michieletto, A.~Cenedese, L.~Zaccarian, and A.~Franchi, ``Nonlinear control
  of multi-rotor aerial vehicles based on the zero-moment direction,'' in
  \emph{IFAC World Congress}, Toulouse, France, Jul. 2017, pp.
  13\,686--13\,691.

\bibitem{MohRezMagPet15}
A.~Mohammadi, E.~Rezapour, M.~Maggiore, and K.~Pettersen, ``Maneuvering control
  of planar snake robots using virtual holonomic constraints,'' \emph{IEEE
  Transactions on Control Systems Technology}, vol.~24, no.~3, pp. 884--899,
  2016.

\bibitem{OttDieAlb15}
C.~Ott, A.~Dietrich, and A.~Albu-Sch{\"a}ffer, ``Prioritized multi-task
  compliance control of redundant manipulators,'' \emph{Automatica}, vol.~53,
  pp. 416--423, 2015.

\bibitem{SassanoEJC16}
T.~Passenbrunner, M.~Sassano, and L.~Zaccarian, ``Optimality-based dynamic
  allocation with nonlinear first-order redundant actuators,'' \emph{European
  Journal of Control}, pp. 33--40, 2016.

\bibitem{PleGriWesAbb03}
F.~Plestan, J.~Grizzle, E.~Westervelt, and G.~Abba, ``Stable walking of a
  7-{DOF} biped robot,'' \emph{IEEE Transactions on Robotics and Automation},
  vol.~19, no.~4, pp. 653--668, 2003.

\bibitem{RozMag14}
A.~Roza and M.~Maggiore, ``A class of position controllers for underactuated
  {VTOL} vehicles,'' \emph{IEEE Transactions on Automatic Control}, vol.~59,
  no.~9, pp. 2580--2585, 2014.

\bibitem{SanGoeTee07}
R.~Sanfelice, R.~Goebel, and A.~Teel, ``{Invariance principles for hybrid
  systems with connections to detectability and asymptotic stability},''
  \emph{IEEE Transactions on Automatic Control}, vol.~52, no.~12, pp.
  2282--2297, 2007.

\bibitem{Sei69}
P.~Seibert, ``On stability relative to a set and to the whole space,'' in
  \emph{Papers presented at the $5^{th}$ Int. Conf. on Nonlinear Oscillations
  (Izdat. Inst. Mat. Akad. Nauk. USSR, 1970)}, vol.~2, Kiev, 1969, pp.
  448--457.

\bibitem{Sei70}
------, ``Relative stability and stability of closed sets,'' in \emph{Sem.
  Diff. Equations and Dynam. Systs. II; Lect. Notes Math.}\hskip 1em plus 0.5em
  minus 0.4em\relax Berlin-Heidelberg-New York: Springer-Verlag, 1970, vol.
  144, pp. 185--189.

\bibitem{SeiFlo95}
P.~Seibert and J.~S. Florio, ``On the reduction to a subspace of stability
  properties of systems in metric spaces,'' \emph{Annali di Matematica Pura ed
  Applicata}, vol. CLXIX, pp. 291--320, 1995.

\bibitem{SeiSua90}
P.~Seibert and R.~Su\'arez, ``Global stabilization of nonlinear cascaded
  systems,'' \emph{Systems \& Control Letters}, vol.~14, no.~5, pp. 347--352,
  1990.

\bibitem{Son89b}
E.~Sontag, ``Remarks on stabilization and input-to-state stability,'' in
  \emph{Proc. of the $28^{\mbox{th}}$ IEEE Conference on decision and Control},
  Tampa, Florida, 1989, pp. 1376 -- 1378.

\bibitem{Tee10}
A.~Teel, ``Observer-based hybrid feedback: a local separation principle,'' in
  \emph{American Control Conference}, Baltimore (MD), USA, June 2010, pp.
  898--903.

\bibitem{ThuConHu16}
J.~Thunberg, J.~Goncalves, and X.~Hu, ``Consensus and formation control on
  {$SE(3)$} for switching topologies,'' \emph{Automatica}, vol.~66, pp.
  109--121, 2016.

\bibitem{Vid80}
M.~Vidyasagar, ``Decomposition techniques for large-scale systems with
  nonadditive interactions: Stability and stabilizability,'' \emph{{IEEE}
  Transactions on Automatic Control}, vol.~25, no.~4, pp. 773--779, 1980.

\bibitem{WesGriCheChoMor07}
E.~Westervelt, J.~Grizzle, C.~Chevallereau, J.~Choi, and B.~Morris,
  \emph{Feedback control of dynamic bipedal robot locomotion}.\hskip 1em plus
  0.5em minus 0.4em\relax CRC press, 2007, vol.~28.

\bibitem{WesGriKod03}
E.~Westervelt, J.~Grizzle, and D.~Koditschek, ``Hybrid zero dynamics of planar
  biped robots,'' \emph{IEEE Transactions on Automatic Control}, vol.~48,
  no.~1, pp. 42--56, 2003.

\end{thebibliography}


\newif\ifbiblio
\bibliofalse

\ifbiblio
\begin{IEEEbiography}
[{\includegraphics[width=1.1in,height=1.25in,clip,keepaspectratio]{maggiore.eps}}]
{Manfredi Maggiore} was born in Genoa, Italy. He received the "Laurea"
degree in Electronic Engineering in 1996 from the University of Genoa
and the PhD degree in Electrical Engineering from the Ohio State
University, USA, in 2000. Since 2000 he has been with the Edward
S. Rogers Sr. Department of Electrical and Computer Engineering,
University of Toronto, Canada, where he is currently Professor. He has
been a visiting Professor at the University of Bologna (2007-2008),
and the Laboratoire des Signaux et Syst\`emes, Ecole CentraleSup\'elec
(2015-2016). His research focuses on mathematical nonlinear control,
and relies on methods from dynamical systems theory and differential
geometry.
\end{IEEEbiography}

\begin{IEEEbiography}
[{\includegraphics[width=1in,height=1.25in,clip,keepaspectratio]{MS.eps}}]
{Mario Sassano}
 was born in Rome, Italy, in 1985.
He received the B.S degree in Automation Systems
Engineering and the M.S degree in Systems and
Control Engineering from the University of Rome
”La Sapienza”, Italy, in 2006 and 2008, respectively.
In 2012 he was awarded a Ph.D. degree by Imperial
College London, UK, where he had been a Research
Assistant in the Department of Electrical and Electronic
Engineering since 2009. Currently he is an
Assistant Professor at the University of Rome ”Tor
Vergata”, Italy. His research interests are focused
on nonlinear observer design, optimal control and differential game theory
with applications to mechatronical systems and output regulation for hybrid
systems. He is Associate Editor of the IEEE CSS Conference Editorial Board
and of the EUCA Conference Editorial Board.
\end{IEEEbiography}

\begin{IEEEbiography}
[{\includegraphics[width=1in,height=1.25in,clip,keepaspectratio]{luca_zacc.eps}}]
{Luca Zaccarian} (SM '09 -- F '16) received the Laurea and the Ph.D. degrees from the University of Roma Tor Vergata (Italy) where has been Assistant Professor in control engineering from 2000 to 2006 and then Associate Professor. Since 2011 he is Directeur de Recherche at the LAAS-CNRS, Toulouse (France) and since 2013 he holds a part-time associate professor position at the University of Trento, Italy. Luca Zaccarian's main research interests include analysis and design of nonlinear and hybrid control systems, modeling and control of mechatronic systems. He is currently a member of the EUCA-CEB, an associate editor for the IFAC journal Automatca. He was a recipient of the 2001 O. Hugo Schuck Best Paper Award given by the American Automatic Control Council.
\end{IEEEbiography}
\fi

\end{document}